\documentclass{amsart}
\usepackage{amssymb,amsthm,amsmath}
\usepackage{amssymb,amsthm,amsmath,hyperref}
\usepackage[none]{hyphenat}
\usepackage[ansinew]{inputenc}
\usepackage[T1]{fontenc}

\newtheorem{theorem}{Theorem}
\newtheorem{proposition}{Proposition}

\begin{document}

\title[Lagrangian Structure]{Lagrangian Structure for Compressible Flow in the Half-space with the Navier Boundary Condition}

\author{Marcelo M. Santos}
\address{Departamento de Matem\'atica, IMECC-UNICAMP\\
({\small Instituto de Matem\'atica, Estat\'\i stica
e Computa\c c\~ao Cient\'\i fica}\\
-{\small Universidade Estadual de Campinas})\\
Rua S\'ergio Buarque de Holanda, 651\\
13083-859 - Campinas - SP, Brazil.}
\email{msantos@ime.unicamp.br}

\author{Edson J. Teixeira}
\thanks{E.J. Teixeira thanks CAPES--Coordena\c c\~ao de Aperfei\c coamento de Pessoal de
N\'\i vel Superior, Brazil, for financial support during this
research.}
\address{Departamento de Matem\'atica, UFV (Universidade Federal de Vi\c cosa)\\
Av. PH. Rolfs, s/n\\
36570-900 - Vi\c cosa - MG, Brazil.}
\email{edson.teixeira@ufv.br}

\maketitle

\begin{abstract}
We show the uniqueness of particle paths of a velocity field, which solves the compressible isentropic Navier-Stokes equations in the half-space $\mathbb{R}_+^3$ with the Navier boundary condition. In particular, by means of energy estimates and the assumption of small energy we prove that the velocity field satisfies the necessary regularity needed to prove the uniqueness of particle paths.
\end{abstract}

{\it 2010 Mathematics Subject Classification}. 35Q30, 76N10, 35Q35, 35B99.

{\it Keywords and phrases}. Navier-Stokes equations; lagrangian structure; Navier boundary condition.

\section{Introduction}
\label{introduction}

This paper concerns the lagrangian structure of the solution obtained by D. Hoff
\cite{HalfSpace} to the Navier-Stokes system for compressible fluids in the half-space
$\mathbb{R}_+^3=\{x=(x_1,x_2,x_3)\in\mathbb{R}^3\, ;\, \ x_3>0\}$ with the Navier boundary condition.
We follow the approach of \cite{Santos}. Due to the presence of the boundary we analyze and show new estimates. For instance, to estimate the $L^p$ norm of the second derivative of a part of the velocity field, which is denoted by $u_{F,\omega}$, we need to consider a {\em singular kernel} on $\partial\mathbb{R}^3_+$, which we deal with the help of Agmon-Douglis-Nirenberg's theorem (Theorem \ref{Singular_fronteira} below). In fact, this part of the velocity field satisfies a boundary value problem in the half-space (see \eqref{bvp1}) for which we use the explicit formulas given by the Green's functions for the half-space with Dirchlet and Neumann boundary conditions (see \eqref{solution dirichlet} and \eqref{solution neumann}). The half-space have several properties we use. In addition to the explicit formulas for Green's functions, it has a {\em strong m-extension
operator}. This property implies that several classical inequalities on $\mathbb{R}^n$ holds also
on $\mathbb{R}^n_+$. In particular, it is very useful the imbedding inequality \eqref{limitacao} and 
the interpolation inequality \eqref{des_interpolacao}, which we can infer from the similar inequalities on $\mathbb{R}^n$. These and other results we shall need are explained in details in Section \ref{preliminaries}. The crucial result, as it is in \cite{Santos}, is the estimates \eqref{expoente1} and \eqref{expoente2} stated in Theorem \ref{main 1}. To show this inequalities with the presence of the boundary (in \cite{Santos} it is considered only the initial value problem) we use the results displayed in Section \ref{preliminaries} and arguments in the papers \cite{Hoff1995,Hoff2002,HalfSpace,Xiangdi}. In particular, to prove Proposition 
\ref{prop3} and Theorem \ref{theorem 2-s} we use some arguments in the proof of \cite[Lemma 3.3]{Xiangdi}.  

Let us describe in more details the results we show in this paper.
First, for the reader convenience, let us recall the solution obtained in \cite{HalfSpace}. Consider the Navier-Stokes equations
\begin{equation}
\label{equations int}
\left\{\begin{array}{l}
\rho_t+\operatorname*{div}(\rho u)=0\\
(\rho u^j)_t+\operatorname*{div}(\rho u^ju)+P(\rho)_{x_j}=\mu\Delta u^j+\lambda \operatorname*{div} u_{x_j}+\rho f^j, \quad j=1,2,3
\end{array}\right.
\end{equation}
for $x\in\mathbb{R}_+^3$ and $t>0$ with the Navier boundary condition
\begin{equation}
\label{boundary int}
u(x,t)=K(x)(u^1_{x_3}(x,t),u^2_{x_3}(x,t),0),
\end{equation}
for $x=(x_1,x_2,0)\in \partial\mathbb{R}_+^3$, $t>0$, and initial condition
\begin{equation}
\label{initial}
(\rho,u)\big|_{t=0}=(\rho_0,u_0),
\end{equation}
where $\rho$ and $u=(u^1,u^2,u^3)$ are, respectively, the unknowns density and velocity vector field of the fluid modeled by these equations.
$P(\rho)$ is the pressure function, $f=(f^1,f^2,f^3)$ is a given external force density, $\mu$ and $\lambda$ are constant viscosities, and $K$ is a smooth and strictly positive function.
Assume that these functions satisfy the following conditions:
\ for fixed $\tilde\rho, \bar\rho$ such that $0<\tilde\rho<\bar\rho$,
\begin{equation}
P\in C^2([0,\bar\rho]), \ P(0)=0, \ P'(\tilde\rho)>0, \
(\rho-\tilde\rho)[P(\rho)-P(\tilde\rho)]>0, \rho\neq\tilde\rho, \rho\in[0,\bar\rho];
\label{pressure}
\end{equation}
\begin{equation}
\mu>0, \quad 0<\lambda<5\mu/4;
\label{viscosities}
\end{equation}
\begin{equation}
K\in (W^{2,\infty}\cap W^{1,3})(\mathbb{R}^2), \quad K(x)\geq \underline{K}>0,
\label{k-condition}
\end{equation}
for some constant $\underline{K}>0$;
\begin{equation}
\int_{\mathbb{R}_+^3}\left[\dfrac{1}{2}\rho_0|u_0|^2+G(\rho_0)\right]dx \le C_0
\label{initial_energy}
\end{equation}
and
\begin{equation}
\sup\limits_{t\geq 0}\|f(.,t)\|_2  +  \int_0^\infty (\|f(.,t)\|_2+\sigma^7\|\nabla f(.,t)\|_4)dt
  +  \int_0^\infty\int_{\mathbb{R}_+^3}(|f|^2+\sigma^5|f_t|^2)dxdt \le C_f,
\label{force_energy}
\end{equation}
where $C_0$ and $C_f$ are positive numbers sufficiently small,
$G(\rho):=\rho\int_{\tilde\rho}^{\rho}\dfrac{P(s)-P(\tilde\rho)}{s^2}ds$ and $\sigma(t):=\min\{t,1\}$, and; the quantitity
\begin{equation}
M_q:=\int_{\mathbb{R}_+^3}\rho_0|u_0|^q+\sup_{t>0}\|f(\cdot,t)\|_q
+\int_0^\infty\int_{\mathbb{R}_+^3}|f|^qdxdt
\label{Mq}
\end{equation}
is finite, where $q>6$ and satisfies
\begin{equation}
\frac{(q-2)^2}{4(q-1)}<\frac{\mu}{\lambda}.
\label{q}
\end{equation}
Here and throughout the paper, $\|\cdot\|_p$ stands for the $L^p$ norm in $\mathbb{R}_+^n$.

Under these conditions, D. Hoff \cite[Theorem 1.1]{HalfSpace} established the existence of a {\lq\lq}small energy{\rq\rq} ($C_0, C_f$ are sufficiently small) weak solution $(\rho,u)$ to
\eqref{equations int}-\eqref{initial} as follows. Given a positive number $M$ (not necessarily small)
and given $\bar\rho_1\in (\tilde\rho,\bar\rho)$, there are positive numbers $\varepsilon$ and $\bar C$
depending on $\tilde\rho,\bar\rho_1,\bar\rho,P,\lambda,\mu,q,M$ and on the function $K$, and there
is a positive universal constant $\theta$, such that, if
$0\le\inf_{\mathbb{R}_+^3}\rho_0\le\sup_{\mathbb{R}_+^3}\rho_0\le\bar\rho_1$, $C_0+C_f\le\varepsilon$ and $M_q\le M$ then there is a weak solution $(\rho,u)$ to \eqref{equations int}-\eqref{initial} satisfying
(among other properties): the functions $u$, $F=(\lambda+\mu) \operatorname*{div} u-P(\rho)+P(\tilde\rho)$ (the
{\em effective viscous flow}) and $\omega^{j,k}=u_{x_k}^{j}-u_{x_j}^{k}, \, j,k=1,2,3$
($\omega=(\omega^{j,k})$ is the vorticity matrix) are H\"older continuous in
$\overline{\mathbb{R}_+^3}\times[\tau,\infty)$, for any $\tau>0$;
$C^{-1}\inf\rho_0\leq \rho\leq\bar\rho$ a.e. and;
\begin{eqnarray}
&\sup\limits_{t>0}\int_{\mathbb{R}_+^3}[\dfrac{1}{2}\rho(x,t)|u(x,t)|^2
+|\rho(x,t)-\tilde\rho|^2+\sigma(t)|\nabla u(x,t)|^2]dx\\
& +  \int_0^\infty\int_{\mathbb{R}_+^3}\left[|\nabla u|^2
+\sigma^3(t)|\nabla\dot u|^2\right]dxdt\nonumber\\
& \leq  \bar C(C_0+C_f)^\theta,\nonumber
\label{theorem_estimate}
\end{eqnarray}
where $\dot u=u_t+(\nabla u)u$ (the {\em convective derivative} of $u$). In addition, in the case that
$\inf_{\mathbb{R}_+^3}\rho_0>0$, the term $\int_0^\infty\int_{\mathbb{R}_+^3}\sigma|\dot u|^2dxdt$
may be included on the left side of \eqref{theorem_estimate}.

In this paper we show the following results.

\begin{proposition}
\label{decomposition}
Let the assumptions \eqref{pressure}-\eqref{q} be in force. Then the above
vector field $u$ can be written as $u=u_P+u_{F,\omega}$ for some vector fields
$u_{F,\omega},u_P$ satisfying:
\begin{equation}
\|\nabla u_P\|_p\leq C \|P-\tilde P\|_p,
\label{nabla uP}
\end{equation}
\begin{equation}
\|\nabla u_{F,\omega}\|_p
               \leq C(\|F\|_p+\|\omega\|_p+\|P-\tilde P\|_p+\|u\|_p)
\label{primeira_derivada_uFw}
\end{equation}
and
\begin{equation}
\|D^2 u_{F,\omega}\|_p \leq
C(\|\nabla F\|_p+\|\nabla\omega\|_p+\|F\|_p+\|\omega\|_p+\|P-\tilde P\|_p+\|u\|_p),
\label{segunda_derivada_uFw}
\end{equation}
for any $p\in (1,\infty)$, where $C$ is a constant depending on $n$, $p$ and  on
arbitrary positive numbers $\underline{K},\overline{K}$ such that
$\underline{K}\le K \le \overline{K}$.
\end{proposition}

\begin{theorem}
\label{main 1}
Let the assumptions \eqref{pressure}-\eqref{q} be in force. Suppose also that $u_0$ belongs to the
Sobolev space $H^s(\mathbb{R}_+^3)$, for some $s\in [0,1]$, and $\inf_{\mathbb{R}_+^3}\rho_0>0$. Then the above Hoff solution $(\rho,u)$ to the problem \eqref{equations int}-\eqref{initial} satisfies the estimates:
\begin{equation}
\sup\limits_{t>0}\sigma^{1-s}\int_{\mathbb{R}_+^3}|\nabla u|^2dx + \int_0^\infty\int_{\mathbb{R}_+^3}\sigma^{1-s}\rho|\dot u|^2dxdt\leq C(s)(C_0+\|u_0\|_{H^s}+C_f)^\theta
\label{expoente1}
\end{equation}
and
\begin{equation}
\sup\limits_{t>0}\sigma^{2-s}\int_{\mathbb{R}_+^3}\rho|\dot u|^2dx + \int_0^\infty\int_{\mathbb{R}_+^3}\sigma^{2-s}|\nabla\dot u|^2dxdt\leq C(s)(C_0+\|u_0\|_{H^s}+C_f)^\theta,
\label{expoente2}
\end{equation}
where $C(s)$ is a constant depending on the same quantities as $\bar C$ above and on $s$.
\end{theorem}


These estimates, as in \cite{Santos}, imply a lagrangian structure for the Hoff solution. 
 More precisely, 
\ the following theorem, which is similar to Theorem 2.5 of \cite{Santos}, holds true for the Navier-Stokes equations \eqref{equations int} in the half-plahe $\mathbb{R}^3_{+}$ with the Navier boundary condition \eqref{boundary int}:

\begin{theorem} ({\bf cf. \cite[Theorem 2.5]{Santos}.})
\label{main}
Under the hypothesis in Theorem \ref{main 1}, if $s>1/2$ then the following assertions are true.
\begin{enumerate}
 \item[(a)] For each $x\in\overline{\mathbb{R}_+^ 3}$, there exists a unique map
$X(\cdot\,,x)\in C([0,\infty))\cap C^1((0,\infty))$ such that
\begin{equation}
X(t,x)=x+\int_0^t u(X(\tau,x),\tau)d\tau, \quad t\in [0,\infty).
\label{lagrangean_structure}
\end{equation}
\item[(b)] For each $t>0$, the map $x\mapsto X(t,x)$ is a homeomorphism of $\overline{\mathbb{R}}_+^3$ into $\overline{\mathbb{R}}_+^3$, leaving
$\partial\mathbb{R}_+^3$ invariant i.e. $X(t,\partial\mathbb{R}_+^3)\subset\partial\mathbb{R}_+^3$.
\item[(c)] Given $t_1, t_2\geq 0$, the map $X(t_1,x)\mapsto X(t_2,x)$, $x\in \mathbb{R}_+^3$, is Hölder
continuous, locally uniform with respect to $t_1,t_2$. More precisely, given any $T>0$, there exist positive numbers
$C$, $L$ and $\gamma$ such that
$$
|X(t_2,y)-X(t_2,x)|\leq C|X(t_1,y)-X(t_1,x)|^{e^{-LT^\gamma}}
$$
for all $t_1,t_2\in[0,T])$ and $x,y\in \mathbb{R}_+^3$.
\item[(d)] Let $\mathcal{M}$ be a parametrized manifold in $\mathbb{R}_+^3$ of class $C^\alpha$, for some $\alpha\in[0,1)$, and of dimension $k$, where $k=1$ or $2$. Then, for each $t>0$, $\mathcal{M}^t:= X(t,\mathcal{M})$ is also a
parametrized manifold of dimension $k$ in $\mathbb{R}_+^3$, and of class $C^\beta$, where $\beta=\alpha e^{_Lt^\gamma}$, being $L$ and $\gamma$ the same constants in item (c).
\end{enumerate}
\end{theorem}

We shall assume throughout the paper, without loss of generality, that the Hoff solution $(\rho, u)$
to \eqref{equations int}-\eqref{initial} is smooth, since it is the limit of smooth solutions
(see \cite[Proposition 3.2 and \S 4]{HalfSpace}) and all the above estimates can be obtained by passing to the limit from corresponding uniform estimates for smooth estimates (a priori estimates). Notice that by the proof of \cite[Proposition 3.2]{HalfSpace} we have
$\rho(\cdot,t),u(\cdot,t)\in H^\infty(\mathbb{R}^3_+)$ for any $t\ge0$, if all data are {\lq\lq}smooth{\rq\rq}.

\medskip

Considering the Cauchy problem, \ D. Hoff \cite{Hoff2002} established the lagrangian structure in
dimension two with the initial velocity in the Sobolev space $H^s$, for an arbitrary $s>0$, while
D. Hoff and M. Santos \cite{Santos} proved that the velocity field was a lipschitzian vector field,
in dimension two and three, for the initial velocity in $H^s$, with $s>0$ in dimension two and
$s>1/2$ in dimension three, and, as a consequence, assured the lagrangian structure in dimensions
two and three;  \ T. Zhang and D. Fang \cite{Zhang-Fang} obtained the lagrangian structure in dimension two for the viscosity $\lambda=\lambda(\rho)$, depending on the fluid density $\rho$, but with the initial velocity in $H^1(\mathbb{R}^2)$, and \ P.M. Pardo \cite{Pardo} extended the lagrangian structure result obtained in  \cite{Santos} to non isentropic fluids in dimension three, but under the hypothesis that the convective derivative of the specific internal energy is square integrable.

With regards to initial and boundary value problems, D. Hoff and M. Perepelitsa \cite{Hoff-Perepelitsa}
showed, in particular, the lagrangian structure in the half-plane with the initial velocity in $H^1$.

\smallskip

This paper contains two more sections. In Section \ref{preliminaries} we display some preliminaries results we use in the proofs of Proposition \ref{decomposition} and theorems \ref{main 1} and \ref{main}. In Section \ref{proofs} we prove these three results.

\section{Preliminaries}
\label{preliminaries}

One of the main properties of a half-space $\mathbb{R}_+^n$ is the existence of a
{\em strong m-extension operator} $\mathcal{E}$, for any $m\in\mathbb{Z}_+$,
and its explicit construction; see \cite[Theorem 5.19 and its proof]{Adams}.
This property implies that several classical inequalities on $\mathbb{R}^n$ holds also
on $\mathbb{R}^n_+$. In particular, it is very useful the inequality
\begin{equation}
\|u\|_{L^\infty(\Omega)}\leq C(\|u\|_{L^2(\Omega)}+\|\nabla u\|_{L^p(\Omega)})
\label{limitacao}
\end{equation}
which we shall use with $\Omega=\mathbb{R}^3_+$, but it is valid for any domain (open set)
$\Omega$ in $\mathbb{R}^n$ that has a {\em strong 1-extension operator} $\mathcal{E}$ that maps $C^1(\Omega)$ into $C^1(\mathbb{R}^n)$ and a {\em simple (0,p)-extension operator} $\mathcal{E}_0$ such that
$\nabla\circ\mathcal{E}=\mathcal{E}_0\circ\nabla$ on $C^1(\Omega)$,
$p>n$, and for all $u\in C^1(\Omega)$ such that $u\in L^2(\Omega)$ and
$\nabla u\in L^p(\Omega)$, being $C$ a constant depending on $n$ and $p$.
Indeed, by the proof of Morrey's inequality \cite[p. 282]{Evans} and the above extension properties, it easy to see that for such
functions $u$'s we have
$\|u\|_{L^\infty(\Omega)} \le \|\mathcal{E}(u)\|_{L^\infty(\mathbb{R}^n)} \le C(\|\mathcal{E}(u)\|_{L^2(\mathbb{R}^n)}+\|\nabla \mathcal{E}(u)\|_{L^p(\mathbb{R}^n)})\le C(\|u\|_{L^2(\Omega)}+\|\nabla u\|_{L^p(\Omega)})$, with possibly different constant $C$'s \ ($C=C(n,p)$).

Actually, many results in this paper certainly hold true for domains in $\mathbb{R}^n$ with the above extension properties and a nice boundary (particularly to which we can assure the existence of the Green function; see below) but we restrict ourselves here to the half-space $\mathbb{R}_+^3$ (although some facts in this Section we discuss on $\mathbb{R}_+^n$,
for a general $n$, since it makes no relevant difference to particularize them to the case $n=3$).

Another very useful inequality is the interpolation inequality
\begin{equation}
\|u\|_{L^p(\mathbb{R}^3_+)}\leq C \|u\|_{L^2(\mathbb{R}^3_+)}^{(6-p)/2p}\|\nabla u\|_{L^2(\mathbb{R}^3_+)}^{(3p-6)/2p},
\label{des_interpolacao}
\end{equation}
which holds for any function $u$ in the Sobolev space $H^1(\mathbb{R}^3_+)$, with $p\in [2,6]$ and $C$ being a constant depending on $p$. Notice that this inequality can be obtained from the same inequality in $\mathbb{R}^3$, similarly as we showed above the inequality \eqref{limitacao}.

In $\mathbb{R}^n_+$ we also have the explicit formula for the Green function, with homogeneous Dirichlet and Neumann boundary conditions, given respectively by
(see e.g. \cite[p. 121]{Gilbarg}),
\begin{equation}
\label{green}
G_D(x,y):=\Gamma(x-y)-\Gamma(x-y^*) \mbox{ and } G_N(x,y):=\Gamma(x-y)+\Gamma(x-y^*),
\end{equation}
for $x,y\in\overline{\mathbb{R}_+^n},\, x\not=y$, where $\Gamma$ is the fundamental solution of the laplacian operator in $\mathbb{R}^n$ 
 and $y^*=(y^*_1,\cdots,y^*_n)$ is the reflection point of $y=(y_1,\cdots,y_n)\in\overline{\mathbb{R}_+^n}$ through the boundary $\partial\mathbb{R}^n_+$, i.e. $y^*_j=y_j$ for
$j=1,\cdots, n-1$ and  $y_n^*=-y_n$. A basic fact we shall use is that the operator
$$
g\mapsto\int_{\mathbb{R}_+^n}\nabla_x G(x,y)g(y)dy,
$$
which we shall denote by $\nabla G*g$, where $G=G_D$ or $G_N$, maps the space $L^p(\mathbb{R}_+^n)\cap L^\infty(\mathbb{R}_+^n)$, for
$1\leq p <n$, continuously into the space of bounded log-lispchitzian functions in $\mathbb{R}_+^n$, i.e. the space of continuous functions $g$ in $\mathbb{R}_+^n$ such that
\begin{equation}
\label{LL}
\|g\|_{LL}\equiv \|g\|_{LL(\mathbb{R}_+^n)}:=\sup_{x\in\mathbb{R}_+^n}|g(x)| + \langle g\rangle_{LL}<\infty,
\end{equation}
where $\langle g\rangle_{LL}:=\sup_{x,y\in\mathbb{R}_+^n;\,0<|x-y|\le 1}\frac{|g(x)-g(y)|}{|x-y|(1-\log|x-y|)}$.
More precisely, if $g\in L^p(\mathbb{R}_+^n)\cap L^\infty(\mathbb{R}_+^n)$, and $1\leq p <n$, then
\begin{equation}
\label{conv fund}
\|\nabla G* g\|_{LL(\mathbb{R}_+^n)}\leq C(\|g\|_{L^{p}(\mathbb{R}_+^n)}+\|g\|_{L^{\infty}(\mathbb{R}_+^n)})
\end{equation}
where $C$ is a constant depending on $n$ and $p$. This follows from the similar result
for $\nabla\Gamma* g$ (see e.g. \cite[Lema 1.3.9]{Pardo}) and the extension (simple $0$-extension) property of $\mathbb{R}_+^n$.
Indeed, denoting by  $\tilde g$ the extension of $g$ to $\mathbb{R}^n$ by reflection through $\partial\mathbb{R}_+^n$
(i.e. $\tilde g(y):= g(y^*)$ when $y_n<0$), in the case
$G(x,y)=G_N(x,y)=\Gamma(x-y)+\Gamma(x-y^*)$ 
 we have
$\nabla G*g=\nabla\Gamma*\tilde g$, where the last $*$ stands for the classical convolution product
in $\mathbb{R}^n$. Then
$\|\nabla G*g\|_{LL(\mathbb{R}_+^n)}=\|\nabla\Gamma*\tilde g\|_{LL(\mathbb{R}^n)}\leq  C(\|\tilde g\|_{L^p(\mathbb{R}^n)}+\|\tilde g\|_{L^\infty(\mathbb{R}^n)})
 \leq  2C (\|g\|_{L^p(\mathbb{R}_+^n)}+\|g\|_{L^\infty(\mathbb{R}_+^n)})$. 
 Regarding $G(x,y)=G_D(x,y)=\Gamma(x-y)-\Gamma(x-y^*)$, it is easy to see that $\nabla G*g = \nabla\Gamma *\tilde g-2\int_{\mathbb{R}_+^n}\nabla \Gamma(x-y^*) g(y)dy$, so we obtain
\eqref{conv fund} similarly, since the last integral has a regular kernel.

Certainly the estimate \eqref{conv fund} is valid for more general domains than
$\mathbb{R}_+^n$ (e.g. the 3d domain in \cite{3ddomain}), since we can obtain it using only the Green function properties.

To estimate solutions of \eqref{equations int}-\eqref{initial} using norms in $H^s$, $0<s<1$,
we shall need to use some interpolations theory, since the space $H^s$ is the interpolation space
$(L^2,H^1)_{s,2}$ (see e.g. \cite{Triebel}). The interpolation Stein-Weiss' theorem
\cite[p. 115]{Bergh} will be also very important to us.

\smallskip

One of the ideas in the analysis of D. Hoff in e.g. \cite{Hoff2002} is to decompose the velocity field $u$ as the sum of two terms, $u_{F,\omega}$ and $u_P$, being the term $u_{F,\omega}$ related to the distinguished quantity
$F=(\lambda+\mu) \operatorname*{div} u-P(\rho)+P(\tilde\rho)$ and to the vorticity matrix
$\omega^{j,k}=u_{x_k}^{j}-u_{x_j}^{k}$, and $u_P$ related to the
fluid pressure $P$. In subsection \ref{proof of decomposition} we exhibit a similar decomposition.
In \cite{Santos}, the vector field $u_P$ is log-lipschitzian with respect to
the spatial variable, with the log-lipschitz norm $\|u_P(\cdot,t)\|_{LL}$ (see \eqref{LL})
locally integrable with respect to $t$, while $u_{F,\omega}$ is a lipschitzian vector field with respect to the spatial variable, with the Lipschitz norm
\begin{equation}
\|u_{F,\omega}\|_{\mbox{\tiny Lip}}\equiv \sup_{x\in\mathbb{R}^3_+}|u_{F,\omega}(x,t)|+\sup_{x,y\in\mathbb{R}^3_+;x\not=y}|u_{F,\omega}(x,t)-u_{F,\omega}(y,t)|/|x-y|
\end{equation}
also locally integrable (the more difficult part to show) with respect to $t$.
Here, this facts are also true, and we have extra difficulties to show them, due to the presence of the boundary.
For instance, to estimate the $L^p$ norm of $D^2u_{F,\omega}$ we need to consider a {\em singular kernel} on $\partial\mathbb{R}^3_+$, which we deal with the help of the following theorem due to Agmon, Douglis and Nirenberg \cite[Theorem 3.3]{Niremberg} (\cite[Theorem II.11.6]{Galdi}).
\begin{theorem}
\label{Singular_fronteira}
Let $p\in (1,\infty)$ and
$\kappa:\overline{\mathbb{R}_+^n}\equiv\mathbb{R}^{n-1}\times [0,\infty)-\{(\mathbf{0},0)\}\to\mathbb{R}$
given by $\kappa(x,x_n)=w(\frac{(x,x_n)}{|(x,x_n)|})/|(x,x_n)|^{n-1}$, where
 $w$ is a continuous function on $\overline{\mathbb{R}_+^n}\cap\mathbb{S}^{n-1}$, H\"older continuous on $\mathbb{S}^{n-1}\cap\{x_n=0\}$ and satisfying $\int_{\mathbb{S}^{n-1}}w(x,0)dx=0$. Assume that $\kappa$ has continuous partial derivatives
$\partial_{x_i}\kappa, i=1,2,...,n$, $\partial_{x_n}^2\kappa$ in $\mathbb{R}_+^n$ which are bounded by a constant $c$ on $\mathbb{R}_+^n\cap \mathbb{S}^{n-1}$.
Then, for any function $\phi\in L^p(\partial\mathbb{R}_+^n)$ that has finite seminorm
$\langle \phi\rangle_{1-1/p,p}\equiv (\int_{\partial\mathbb{R}_+^n}\int_{\partial\mathbb{R}_+^n}\dfrac{|\phi(x)-\phi(y)|^p}{|x-y|^{n-2+p}}dxdy)^{1/p}$,
the function $\psi(x,x_n):=\int_{\partial\mathbb{R}_+^n}\kappa(x-y,x_n)\phi(y)dy$
belongs to $L^p(\mathbb{R}_+^n)$ and $\|\nabla\psi\|_{L^p(\mathbb{R}_+^n)}\leq Cc\langle\phi\rangle_{1-1/p}$, where $C$ is a constant depending on $n$ and $p$.
\end{theorem}

In fact, the coordinates of the vector fields $u_{F,\omega}, u_P$ here, described in
\S \ref{proof of decomposition}, satisfy boundary value problems for Poisson equations of the form
$-\Delta v=g_{x_j}$ in the half-space $\mathbb{R}^3_+$, for some function $g$, with Neumann or Dirichlet boundary condition. In this regard, we shall use the formulas
\begin{equation}
\begin{array}{rl}
v(x)&=-\int_{\mathbb{R}_+^n}G_D(x,y)g(y)_{y_j}dy
        -\int_{\mathbb{R}^{n-1}} G_D(x,y)_{y_n}h(y)dy\\
        &= \ \ \ \int_{\mathbb{R}_+^n}G_D(x,y)_{y_j}g(y)dy
                     -\int_{\mathbb{R}^{n-1}} G_D(x,y)_{y_n}h(y)dy,
\end{array}
\label{solution dirichlet}
\end{equation}
\begin{equation}
\begin{array}{rl}
v(x)&=-\int_{\mathbb{R}_+^n}G_N(x,y)g(y)_{y_j}dy
           -\int_{\mathbb{R}^{n-1}}G_N(x,y)h(y)dy\\
      &=\   \ \ \int_{\mathbb{R}_+^n}G_N(x,y)_{y_j}g(y)dy
           -\int_{\mathbb{R}^{n-1}}G_N(x,y)h(y)dy,
\end{array}
\label{solution neumann}
\end{equation}
for the solutions of the the boundary value problems
\begin{equation}
\left\{\begin{array}{cl}
-\Delta v=g_{x_j} & \textrm{ in } \mathbb{R}_+^n\\
 v=h & \textrm{ on } \mathbb{R}^{n-1},
\end{array}\right.
\label{dirichlet}
\end{equation}
\begin{equation}
\left\{\begin{array}{ll}
-\Delta v=g_{x_j} & \textrm{ in } \mathbb{R}_+^n\\
-v_{x_n}=h & \textrm{ on } \mathbb{R}^{n-1},
\end{array}\right.
\label{neumann}
\end{equation}
respectively, for $j=1,\cdots,n$, and $g\in H^m(\mathbb{R}_+^n), h\in H^m(\mathbb{R}^{n-1})$ with a sufficiently large $m$, where $G_D$ and $G_N$ are the Green functions in $\mathbb{R}_+^n$ with the homogeneous Dirichlet and Neumann boundary conditions, respectively (see \eqref{green}) and in the case $j=n$
we can assume $g|\mathbb{R}^{n-1}=0$, without loss of generality. \ Notice that, extending $g$ to a
function $\tilde g\in H^m(\mathbb{R}^n)$ (see \cite[Theorem 5.19]{Adams}) we can write the integral
$w(x):=\int_{\mathbb{R}_+^n}G(x,y)_{y_j}g(y)dy$, where $G=G_D,G_N$, in
\eqref{solution dirichlet},\eqref{solution neumann}, as
$w(x)=\int_{\mathbb{R}^n}\Gamma(x-y)_{y_j}\tilde g(y)dy
                -\int_{\mathbb{R}_-^n}\Gamma(x-y)_{y_j}\tilde g(y)dy
                \pm \int_{\mathbb{R}_+^n}\Gamma(x-y^*)_{y_j}g(y)dy$,
being the last two integrals harmonic functions in $\mathbb{R}_+^n$, since their kernels
are regular, for $x\in\mathbb{R}_+^n$. The first integral satisfies the equation
$-\Delta w=\tilde g_{x_j}$ in $\mathbb{R}^n$ in the classical sense (cf. e.g. \cite[\S 2.2, Theorem 1]{Evans}
where the condition of the right hand side of the Poisson equation having support compact can be replaced by the condition of being in $H^m(\mathbb{R}^n)$ for a sufficiently large $m$, as can be seen by checking the proof).  Besides, we also can write
$w(x)=\int_{\mathbb{R}_+^n}\Gamma(x-y)_{y_j} g(y)dy \pm \int_{\mathbb{R}_-^n}\Gamma(x-y)_{y_j} g(y*)dy
=\int_{\mathbb{R}^n}\Gamma(x-y)_{y_j} [\bar g(y) \pm \bar{\bar g}(y)]dy$,
where $\bar g$ and $\bar{\bar g}$ denote, respectively, the extensions by zero to $\mathbb{R}^n$ of $g$ and
$g(y*)$, from which, by using that the second derivative $\Gamma_{y_iy_j}$ of the fundamental solution for the laplacian in $\mathbb{R}^n$ is a {\em singular kernel}, we can infer the estimate
\begin{equation}
\|\nabla_x \int_{\mathbb{R}_+^n}G(x-y)_{y_j} g(y)dy\|_p\leq C \|g\|_p,
\label{1nabla w}
\end{equation}
for any $p\in (1,\infty)$, where $G=G_D,G_N$ and $C$ is a constant depending on $n$ and $p$.
On the other hand, writing $w(x)=-\int_{\mathbb{R}_+^n}G(x,y)g_{y_j}(y)dy$, by the same argument, we have also the estimate
\begin{equation}
\|D^2_x\int_{\mathbb{R}_+^n}G(x,y)_{y_j}g(y)dy\|_p\leq C \|\nabla g\|_p,
\label{2nabla w}
\end{equation}
for $p,G,C$ as in \eqref{1nabla w}. \ Regarding the boundary integrals (i.e. over $\mathbb{R}^{n-1}$) in
\eqref{solution dirichlet} and\eqref{solution neumann}, we observe that the
function $x\mapsto \int_{\mathbb{R}^{n-1}} G_D(x,y)_{y_n}h(y)dy$ defines a classical
solution to \eqref{dirichlet}, with $g=0$, if only $h$ is continuous and
bounded, as it is well known, and as for $\int_{\mathbb{R}^{n-1}}G_N(x,y)h(y)dy$, it
defines a solution to \eqref{neumann}, with also $g=0$, if $h$ is continuous and
have a nice decay at infinity (e.g. $h\in H^m(\mathbb{R}^{n-1})$ for some large $m$);
see \cite{Love,Armitage}. Besides, using Theorem \ref{Singular_fronteira},
we have the estimate
\begin{equation}
\|D^2_x\int_{\mathbb{R}^{n-1}}G_N(x,y)h(y)dy\|_p \le C \langle h\rangle_{1-1/p,p}
\le C \|\nabla\tilde h\|_p,
\label{2nabla neumann}
\end{equation}
for any $p\in (1,\infty)$, where $\tilde h$ is any extension of $h$ in $H^1(\mathbb{R}_+^n)$, $C$ is a constant depending on $n$ and $p$, and for the last inequality we used \cite[Theorem II.10.2]{Galdi}.  \ It is interesting to notice
that the problem for the laplacian equation $\Delta v=0$ in $\mathbb{R}_+^n$ with the
boundary condition $Kv_{x_n}=v$ on $\partial\mathbb{R}_+^n$, which is required for the
coordinates $u_1$ and $u_2$ of the vector field $u$ in the Navier boundary
condition \eqref{boundary int}, can be reduced to the boundary value problem
\eqref{neumann} with homogeneous boundary condition (i.e. with $h=0$ in
\eqref{neumann}) through the change of variable $V=\varphi v$
(suggested to us by D. Hoff) where $\varphi$ is a suitable function coinciding
with $e^{-K^{-1}x_n}$ on $\partial\mathbb{R}_+^n$. From this observation,
using \eqref{1nabla w}, \eqref{2nabla w} and that $\|G_N\ast v\|_p\le C\|v\|_p$, it is possible to show the estimates
\begin{equation}\begin{array}{rl}\label{navier/estimates}
\|\nabla v\|_p\le C\|v\|_p, \quad \|D^2 v\|_p\le C\|\nabla v\|_p
\end{array}\end{equation}
for the solution to the problem
$\Delta v=0 \mbox{ in } \mathbb{R}_+^n, \ Kv_{x_n}=v \mbox{ on }\partial\mathbb{R}_+^n$ and any
$p\in (1,\infty)$, where $C$ is as in \eqref{segunda_derivada_uFw}.
 \ Finally, we observe that the solutions to the problems \eqref{dirichlet} and \eqref{neumann} given,
respectively, by \eqref{solution dirichlet} and \eqref{solution neumann}, are unique in the space
$L^p(\mathbb{R}_+^n)\cap L^\infty(\mathbb{R}_+^n)$, for an arbitrary $p\in [1,\infty)$. Indeed, if $v$ is a solution of
\eqref{dirichlet} in $L^p(\mathbb{R}_+^n)\cap L^\infty(\mathbb{R}_+^n)$ with $g=h=0$, extending it to $\mathbb{R}^n$ as an odd function
with respect to $x_n$, we obtain an integrable harmonic function (in the sense of the distributions) and bounded, in $\mathbb{R}^n$, then, by Liouville's theorem, $v=0$. We can conclude the same result with
respect to \eqref{neumann} by taking instead an even extension with respect to $x_n$.

\section{Proofs}
\label{proofs}

In this section we prove Propostion \ref{decomposition} and theorems \ref{main 1} and \ref{main}.

\subsection{Proof of Proposition \ref{decomposition}}
\label{proof of decomposition}

Similarly to \cite[(2.28)]{Hoff-Perepelitsa},
we define $u_P$ as the solution of the boundary value problem
\begin{equation}
\left\{\begin{array}{cl}(\lambda+\mu)\Delta u_P=\nabla(P-\tilde P), & \textrm{ in } \mathbb{R}_+^3\\
u_P^3=(u_P^2)_{x_3}=(u_P^1)_{x_3}=0, & \textrm{ on }\partial\mathbb{R}_+^3,
\end{array}\right.
\label{up}
\end{equation}
i.e.
\begin{equation}
(\lambda+\mu)u_P^j(x) = \int_{\mathbb{R}_+^3}G_N(x,y)_{y_j}(P-\tilde P)(y)dy
=\int_{\mathbb{R}_+^3}(\Gamma(x-y)+\Gamma(x-y^*))_{y_j}(P-\tilde P)(y)dy, 
\label{formula_u_P_j12}
\end{equation}
for $j=1,2$, $x\in\mathbb{R}_+^3$, and
\begin{equation}
(\lambda+\mu)u_P^3(x) =  \int_{\mathbb{R}_+^3}G_D(x,y)_{y_3}(P-\tilde P)(y)dy
=\int_{\mathbb{R}_+^3}(\Gamma(x-y)-\Gamma(x-y^*))_{y_3}(P-\tilde P)(y)dy,
\label{formula_u_P_j3}
\end{equation}
$x\in\mathbb{R}_+^3$;
see \eqref{solution dirichlet} and \eqref{solution neumann}. By \eqref{1nabla w}, we have the estimate
\begin{equation}
\|\nabla u_P^j\|_p\leq C \|P-\tilde P\|_p, \quad j=1,2,3,
\label{upj}
\end{equation}
for any $p\in (1,\infty)$, with $C$ being a constant depending on $n$ and $p$.

Next we define $u_{F,\omega}$ as $u_{F,\omega}=u-u_P$. Using \eqref{upj}, it follows that
\begin{equation}\label{nabla ufw depending on nabla u}
\|\nabla u_{F,\omega}\|_p\leq C (\|\nabla u\|_p+\|P-\tilde P\|_p)
\end{equation}
for any $p\in (1,\infty)$, with $C$ being a constant depending on $n$ and $p$.
On the other hand, by the definitions of $u_P$, the Navier boundary condition \eqref{boundary int}, and observing that the the
momemtum equation (second equation in \eqref{equations int}) can be written in terms of the
{\em effective viscous flow} $F$ and of the vortex matrix $\omega$ as
$(\lambda+\mu)\Delta u^ j =F_{x_j}+(P-\tilde P)_{x_j}+(\lambda+\mu)\sum_{k=1}^3\omega_{x_k}^{j,k}$,
we have that $u_{F,\omega}$ satisfies the boundary value problem
\begin{equation}
\label{bvp1}
\left\{\begin{array}{cl}
(\lambda+\mu)\Delta u_{F,\omega}
=\nabla F+(\lambda+\mu)\sum_{k=1}^3\omega_{x_k}^{\cdot,k},
& \textrm{ in } \mathbb{R}_+^3\\
u_{F,\omega}^3=0, \ \ (u_{F,\omega}^j)_{x_3}=K^{-1}u^j, \ j=2,3, & \textrm{ on }\partial\mathbb{R}_+^3.\end{array}\right.
\end{equation}
Then by \eqref{1nabla w}, \eqref{2nabla w} and \eqref{2nabla neumann}, we have
\begin{equation}\label{2nabla ufw depending on nabla u}
\|D^2 u_{F,\omega}\|_p\leq C(\|\nabla F\|_p+\|\nabla\omega\|_p+\|\nabla u\|_p),
\end{equation}
for $p$ and $C$ as in \eqref{segunda_derivada_uFw}. Now, the velocity field $u$ satisfies the boundary value problem
\begin{equation}
\label{bvp2}
\left\{\begin{array}{cl}
(\lambda+\mu)\Delta u
=\nabla F+(\lambda+\mu)\sum_{k=1}^3\omega_{x_k}^{\cdot,k}+\nabla (P-\tilde P),
& \textrm{ in } \mathbb{R}_+^3\\
u^3=0, \ \ u^j_{x_3}=K^{-1}u^j, \ j=2,3, & \textrm{ on }\partial\mathbb{R}_+^3.\end{array}\right.
\end{equation}
Then, by \eqref{1nabla w} and \eqref{navier/estimates}, we have
the estimate \cite[Lemma 2.3, item (b)]{HalfSpace}
\begin{equation}
\|\nabla u\|_p\le C (\|F\|_p+\|\omega\|_p+\|P-\tilde P\|_p+\|u\|_p)
\label{1nabla uFw}
\end{equation}
where $p$ and $C$ are as in \eqref{segunda_derivada_uFw}.
By \eqref{upj}, \eqref{nabla ufw depending on nabla u}, \eqref{2nabla ufw depending on nabla u} and \eqref{1nabla uFw}, we conclude the proof of Theorem \ref{decomposition}. \ \fbox

\subsection{Proof of Theorem \ref{main 1}}
\label{proof of main 1}

To prove \eqref{expoente1}, following \cite{Hoff2002} and \cite{Hoff-Perepelitsa}, we write $u=v+w$,
where $v$ is the solution of a linear homogeneous system with initial condition
$v|_{t=0}=u_0$ and $w$ is the solution of a linear nonhomogeneous
system with initial homogeneous initial condition. More precisely, taking the differential operator $\mathcal{L}\equiv ({\mathcal L}^1, {\mathcal L}^2, {\mathcal L}^3)$ given by
$\mathcal{L}^j(z)=\rho\dot z^j-\mu\Delta z^j-\lambda \operatorname*{div}z_j, \quad j=1,2,3, \quad z = (z^1,z^2,z^3)$,
where $\dot z$ is the convective derivative $z_t + u\nabla z$, we define $v$ and $w$ as the solutions of the following initial boundary value problems
\begin{equation}
\left\{\begin{array}{ll}
\mathcal{L}(v)=0, & \textrm{ in }{\mathbb{R}_+^3} \\
(v^1,v^2,v^3)=K^{-1}(v_{x_3}^1,v_{x_3}^2,0), & \textrm{ on } \partial{\mathbb{R}_+^3}\\
v(.,0)=u_0, &
\end{array}\right.
\label{homogeneo}
\end{equation}
\begin{equation}
\left\{\begin{array}{ll}
\mathcal{L}(w)=-\nabla (P-\tilde P)+\rho f, & \textrm{ in }{\mathbb{R}_+^3} \\
(w^1,w^2,w^3)=K^{-1}(w_{x_3}^1,w_{x_3}^2,0), & \textrm{ on } \partial{\mathbb{R}_+^3}\\
w(.,0)=0. &
\end{array}\right.
\label{nao-homogeneo}
\end{equation}
Then $v$ and $w$ are estimated separately. To estimate $v$ the interpolation theory
is used, since the initial data $u_0$ is in $H^s$ and $H^s$ is the interpolation space
$\left(L^2,H^1\right)_{s,2}$;  see \cite[p. 186 and 226]{Triebel}. We shall  use also the Stein-Weiss' theorem
for $L^p$ spaces with weights \cite[p. 115]{Bergh}. For $w$, the interpolation theory is not
needed, since the initial condition is null. Actually, $w$ satisfies the estimate
\eqref{expoente1} with $s=0$ (equation \eqref{interp_para_w} below).

\begin{proposition}
If $u_0\in H^s(\mathbb{R}_+^3)$, $0\le s\le 1$, then for any positive number $T$ there
is a constant $C$ independent of $(\rho,u),v,w,\rho_0,u_0$ and $f$ such that
\begin{equation}
\sup\limits_{0\le t\le T} \sigma^{1-s}(t)\int_{\mathbb{R}_+^3}|\nabla v|^2dx + \int_0^T\int_{\mathbb{R}_+^3} \sigma^{1-s}(t)\rho|\dot v|^2dxdt\leq C||u_0||_{H^s({\mathbb{R}_+^3})}^2.
\label{interp_para_v}
\end{equation}
\end{proposition}
\begin{proof}
We shall obtain \eqref{interp_para_v} for $s=1$ when $u_0\in L^2(\mathbb{R}_+^3)$ and for $s=0$ when $u_0\in H^1(\mathbb{R}_+^3)$. Then \eqref{interp_para_v} follows by interpolation.

Multiplying the equation $\rho\dot v^j=\mu\Delta v^j+\lambda(\operatorname*{div} v)_j$ by $v_t^j$ and integrating, we obtain
$$
\begin{array}{rl}
 &\int_{\mathbb{R}_+^3}\rho|\dot v|^2dx-\int_{\mathbb{R}_+^3}\rho\dot v^ j u.\nabla v^ j dx=\mu\int_{\mathbb{R}_+^3}\Delta v^ jv_t^ jdx+\lambda\int_{\mathbb{R}_+^3}(\operatorname*{div} v)_jv_t^ jdx\\
=&-\mu\int_{\mathbb{R}_+^3}\nabla v^j.\nabla v_t^jdx+\mu\int_{\partial{\mathbb{R}_+^3}}v_t^j\nabla v^j.\nu dS_x
- \lambda\int_{\mathbb{R}_+^3} (\operatorname*{div} v)(\operatorname*{div} v)_tdx\\
&\ \ \ \ \ \ \ \ \ \ \ \ \ \ \ \ \ \ \ \ \ \ \ \ \ \ \ \ \ \ \ \ \ \ \ \ \ \ \ \ \ \ \ \ \ \ \ \ \ \ \ \ \ \ +\ \lambda\int_{\partial{\mathbb{R}_+^3}}(\operatorname*{div} v)v_t^j\nu^jdS_x\\
=&-\dfrac{\mu}{2}\dfrac{d}{dt}\int_{\mathbb{R}_+^3} |\nabla v|^ 2dx-\dfrac{\lambda}{2}\dfrac{d}{dt}\int_{\mathbb{R}_+^3}|\operatorname*{div} v|^ 2dx+\mu\int_{\partial{\mathbb{R}_+^3}}v_t^ jv_k^ j\nu^ kdS_x\\
=&-\dfrac{1}{2}\dfrac{d}{dt}\left\{\mu\int_{\mathbb{R}_+^3}|\nabla v|^2dx +\lambda\int_{\mathbb{R}_+^3} (\operatorname*{div} v)^2dx+\int_{\partial\mathbb{R}_+^3}\mu K^{-1} |v|^ 2dS_x\right\}.
\end{array}
$$
Then
$$
\begin{array}{rl}
&\dfrac{1}{2}\dfrac{d}{dt}\left(\mu||\nabla v||_2^2+\lambda||\operatorname*{div} v||_2^2+\mu\int_{\partial{\mathbb{R}_+^3}}K^{-1}|v|^2dS_x\right)  +  \int_{\mathbb{R}_+^3}\rho|\dot v|^2dx\\
=&\int_{\mathbb{R}_+^3} \rho \dot v^j(u.\nabla v^j)dx\\
\leq & C(\bar\rho) \left(\int_{\mathbb{R}_+^3}\rho|u|^3dx\right)^{1/3}
\left(\int_{\mathbb{R}_+^3}\rho|\dot v|^2dx\right)^{1/2}\left(\int_{\mathbb{R}_+^3}|\nabla v|^6dx\right)^{1/6}\\
\leq & C(\bar\rho) \|\rho u\|_2^a\|\rho u\|_q^{1-a}
\|\rho\dot v\|_2 \|\nabla v\|_6\\
\leq &C(\bar\rho)(C_0+C_f+M_q)^\theta\|\rho\dot v\|_2\|\nabla v\|_6,
\end{array}
$$
for some $a\in (0,1)$, where $q>6$ and $M_q$ are defined in \eqref{q} and
\eqref{Mq}, $\theta$ is some universal positive constant, and we used \cite[Proposition 2.1]{HalfSpace} and \eqref{theorem_estimate}.

Now defining $\tilde F=(\lambda+\mu)\operatorname*{div} v$ and $\tilde\omega^{j,k}=v_{x_k}^{j}-v_{x_j}^{k}$, we have
$(\lambda+\mu)\Delta v^j=\tilde F_{x_j}+(\lambda+\mu)\tilde\omega_{x_k}^{j,k}$
and, analogously to \cite[Lemma 2.3]{HalfSpace}, it follows the estimates
$$
\|\nabla v\|_p\leq C(\|v\|_p+\|\tilde\omega\|_p+\|\tilde F\|_p),
$$
$$
\|\nabla\tilde F\|_p+\|\nabla\tilde\omega\|_p \leq C(\|\rho \dot v\|_p+\|\nabla v\|_p+\|v\|_p),
$$
for any $p\in (1,\infty)$.
Thus by \eqref{des_interpolacao} and energy estimates we have
$$
\begin{array}{rl}
     &\int_{\mathbb{R}_+^3}\rho\dot v(u.\nabla v^j)dx\\
\leq & C\|\rho\dot v\|_2\left(\|v\|_6+\|\tilde w\|_6+\|\tilde F\|_6\right)\\
\leq & C(C_0+C_f)^\theta\|\rho\dot v\|_2\left(\|\nabla v\|_2+\|\nabla\tilde w\|_2+\|\nabla\tilde F\|_2\right)\\
\leq & C(C_0+C_f)^\theta\|\rho\dot v\|_2\left(\|\nabla v\|_2+\|\rho\dot v\|_2+\|v\|_2\right)\\
= & C(C_0+C_f)^\theta\|\rho\dot v\|_2\|\nabla v\|_2+C(C_0+C_f)^\theta\|\rho\dot v\|_2^2+C(C_0+C_f)^\theta\|\rho\dot v\|_2\|v\|_2\\
\leq & C(C_0+C_f)^\theta\|\nabla v\|_2^2+C(C_0+C_f)^\theta\|\rho\dot v\|_2^2+C\|v\|_2^2\\
= & C(C_0+C_f)^\theta\int_{\mathbb{R}_+^3}|\nabla v|^2dx+C(C_0+C_f)^\theta\int_{\mathbb{R}_+^3}\rho|\dot v|^2dx+C(C_0+C_f)^\theta\int_{\mathbb{R}_+^3}|v|^2dx
\end{array}
$$
Therefore, if $C_0,C_f$ are sufficiently small,
$$
\begin{array}{rl}
&\dfrac{1}{2}\dfrac{d}{dt}(\mu\|\nabla v\|_2^2+\lambda\|\operatorname*{div} v\|_2^2+\mu\int_{\partial{\mathbb{R}_+^3}}K^{-1}|v|^2dS_x)+\int_{\mathbb{R}_+^3}\rho|\dot v|^2\\
\leq & C\int_{\mathbb{R}_+^3} |\nabla v|^2dx+C\int_{\mathbb{R}_+^3} |v|^2dx,
\label{exp_interp}
\end{array}
$$
so integrating in $(0,t)$,
$$
\begin{array}{rl}
&\dfrac{\mu}{2}\int_{\mathbb{R}_+^3} |\nabla v|^2dx+\dfrac{\lambda}{2}\int_{\mathbb{R}_+^3}|\operatorname*{div} v|^2dx  +  \dfrac{\mu}{2}\int_{\partial{\mathbb{R}_+^3}}K^{-1}|v|^2dS_x+\int_0^T\int_{\mathbb{R}_+^3}\rho|\dot v|^2dxds\\
\leq & \dfrac{\mu}{2}\int_{\mathbb{R}_+^3}|\nabla u_0|^2dx+\dfrac{\lambda}{2}\int_{\mathbb{R}_+^3}|\operatorname*{div} u_0|^2dx
 +  \dfrac{\mu}{2}\int_{\partial{\mathbb{R}_+^3}}K^{-1}|u_0|^2dS_x+ C\int_0^T\int_{\mathbb{R}_+^3} |v|^2dxds\\
 \leq & C\|u_0\|_{H^1({\mathbb{R}_+^3})}^2,
\end{array}
$$
if $u_0\in H^1({\mathbb{R}_+^3})$. \ On the other hand, multiplying (\ref{exp_interp}) by $\sigma(t)$, we get
$$
\begin{array}{rl}
&-\dfrac{1}{2}\sigma'\left(\mu\|\nabla v\|_2^2+\lambda\|\operatorname*{div} v\|_2^2
+\mu\int_{\partial{\mathbb{R}_+^3}}K^{-1}|v|^2dS_x\right)\\
& \ \ \ \ \ \ +  \sigma\int_{\mathbb{R}_+^3}\rho|\dot v|^2dx
 +  \dfrac{1}{2}\dfrac{d}{dt}\left(\mu\sigma\|\nabla v\|_2^2+\lambda\sigma\|\operatorname*{div} v\|_2^2+\mu\alpha\sigma\int_{\partial{\mathbb{R}_+^3}}K^{-1}|v|^2dS_x\right)\\
\leq &\sigma C\int_{\mathbb{R}_+^3}|\nabla v|^ 2dx + C\int_{\mathbb{R}_+^3} |v|^2dx,
\end{array}
$$
so integrating in $(0,t)$,
$$
\begin{array}{rl}
&\sigma\dfrac{\mu}{2}\|\nabla v\|_2^2 + \sigma\dfrac{\lambda}{2}\|\operatorname*{div} v\|_2^2 + \sigma\dfrac{\mu}{2}\int_{\partial{\mathbb{R}_+^3}}K^{-1}|v|^2dS_x+\int_0^T\int_{\mathbb{R}_+^3}\sigma\rho|\dot v|^2dxds\\
\leq & \int_0^T\int_{\mathbb{R}_+^3}\sigma'|\nabla v|^2dxds + \int_0^T\int_{\mathbb{R}_+^3}\sigma'|\operatorname*{div} v|^2dxds\\
& \ \ \ +\int_0^T\int_{\partial\mathbb{R}_+^3}\sigma'K^{-1}|v|^2dxdS_x+\sigma\int_0^T\int_{\mathbb{R}_+^3}|\nabla v|^2dx+C\int_0^T\int_{\mathbb{R}_+^3} |\nabla v|^2dxds\\
\leq & C\|u_0\|_2^2,
\end{array}
$$
if $u_0\in L^2({\mathbb{R}_+^3})$. In conclusion, we have the following estimates for $v$,
$$
\sup\limits_{0\leq t\leq T} \int_{\mathbb{R}_+^3} |\nabla v|^2dx +\int_0^T\int_{\mathbb{R}_+^3}\rho|\dot v|^2dxdt \leq C\|u_0\|_{H^1({\mathbb{R}_+^3})}^2,
$$
$$
\sup\limits_{0\leq t\leq T} \sigma(t)\int_{\mathbb{R}_+^3} |\nabla v|^2dx+\int_0^T\int_{\mathbb{R}_+^3} \sigma(t)\rho|\dot v|^2dxdt \leq C\|u_0\|_{L^2({\mathbb{R}_+^3})}^2.
$$
In particular, for any fixed $t>0$, we have that the operator $u_0 \longmapsto  \nabla v$
is linear continuous from $L^2(\mathbb{R}_+^3)$ into $L^2({\mathbb{R}_+^3})$ and from $H^1({\mathbb{R}_+^3})$
into $L^2({\mathbb{R}_+^3})$ with respective norms bounded by
$C\sigma(t)^{-1/2}$ and $C$. Then by interpolation (see \cite[p. 186 and 226]{Triebel}) we obtain
$$
\sup\limits_{0\leq t\leq T} \sigma(t)^{1-s}\int_{\mathbb{R}_+^3} |\nabla v|^2dx\leq C\|u_0\|_{H^s({\mathbb{R}_+^3})}^{2}.
$$
Also, from the above estimates, we have that the operator $u_0 \longmapsto \dot v$ is linear bounded
from $L^2({\mathbb{R}_+^3})$ into $L^2((0,T)\times{\mathbb{R}_+^3},\sigma(t)dtdx)$ and from $H^1({\mathbb{R}_+^3})$
into $L^2((0,T)\times{\mathbb{R}_+^3})$. Then
$$
\int_0^T\int_{\mathbb{R}_+^3} \sigma^{1-s}(t)\rho|\dot v|^2dxdt\leq C\|u_0\|_{H^{s}({\mathbb{R}_+^3})}^2
$$
(see \cite[p. 115]{Bergh}).
\end{proof}

\begin{proposition}
\label{prop3}
For any positive number $T$ there
is a constant $C$ independent of $(\rho,u),v,w,\rho_0,u_0$ and $f$ such that
\begin{equation}
\sup\limits_{0\le t\le T}\int_{\mathbb{R}_+^3}|\nabla w|^2dx
+ \int_0^T\int_{\mathbb{R}_+^3}\rho|\dot w|^2dxdt\leq C(C_0+C_f)^\theta,
\label{interp_para_w}
\end{equation}
for some universal positive constant $\theta$.
\end{proposition}
\begin{proof}
Multiplying (\ref{nao-homogeneo}) by $w_t^j$, summing in $j$ and integrating over $\mathbb{R}_+^3$, we get
$$
\begin{array}{ll}
  &\int_{\mathbb{R}_+^3}\rho|\dot w|^2dx - \int_{\mathbb{R}_+^3}\rho\dot w^ju.\nabla w^jdx
=  -\mu\int_{\mathbb{R}_+^3}(\nabla w^j)(\nabla w^j)_tdx\\
& \ \ \ \ \ \ \ \ + \ \mu\int_{\partial\mathbb{R}_+^3} w_t^j(\nabla w^j).\nu dS(x)
-  \lambda\int_{\mathbb{R}_+^3}(\operatorname*{div} w)(\operatorname*{div} w)_tdx\\
& \ \ \ \ \ \ \ \ + \ \int_{\mathbb{R}_+^3}(P-\tilde P)(\operatorname*{div} w)_tdx
+  \int_{\mathbb{R}_+^3}\rho f^jw_t^jdx,
\end{array}
$$
thence,
$$
\begin{array}{rl}
&\int_{\mathbb{R}_+^3}\rho|\dot w|^2dx  +  \dfrac{d}{dt}(\dfrac{\mu}{2}\int_{\mathbb{R}_+^3}|\nabla w|^2dx + \dfrac{\lambda}{2}\int_{\mathbb{R}_+^3}|\operatorname*{div} w|^2dx-\int_{\mathbb{R}_+^3} (P-\tilde P)\operatorname*{div} w)\\
= & \int_{\mathbb{R}_+^3}\rho \dot w^ju.\nabla w^jdx-\int_{\mathbb{R}_+^3} P_tw_j^jdx
 +  \mu\int_{\partial\mathbb{R}_+^3} w_t^j(\nabla w^j).\nu dS(x)+\int_{\mathbb{R}_+^3} \rho f^jw_t^jdx\\
\equiv & I_1+I_2+I_3+I_4.
\end{array}
$$
Let us estimate each of these integrals $I_1,I_2,I_3,I_4$ separately.

Using estimates for $w$ analogous to those for $u$ in \cite[Lemma 2.3]{HalfSpace} and \eqref{des_interpolacao}, it is possible to show that
$$
\begin{array}{ll}
& I_1  =  \int_{\mathbb{R}_+^3}\rho\dot w^ju.\nabla w^jdx\\
\leq & C(\int_{\mathbb{R}_+^3} \rho|\dot w|^2dx)^{1/2}(\int_{\mathbb{R}_+^3}\rho|u|^3dx)^{1/3}\|\nabla w\|_6\\
 \leq & C(C_0+C_f)^\theta(\int_{\mathbb{R}_+^3}\rho|\dot w|^2dx)^{1/2}(\|\rho\dot w\|_2+\|\nabla w\|_2+\|f\|_2+\|w\|_2+\|P-\tilde P\|_6)\\
  \leq & C(C_0+C_f)^\theta + C(C_0+C_f)^\theta\int_{\mathbb{R}_+^3}\rho|\dot w|^2dx+ C(C_0+C_f)^\theta\int_{\mathbb{R}_+^3}|\nabla w|^2dx.
\end{array}
$$

Writing the identity $(\lambda+\mu)\Delta w^j=\tilde{\tilde F}_{x_j}+(\lambda+\mu)\tilde{\tilde \omega}_{x_k}^{j,k}+(P-\tilde P)_{x_j}$, with $\tilde{\tilde F}=(\lambda+\mu)\operatorname*{div} w-P(\rho)+P(\tilde\rho)$ e $\tilde{\tilde \omega}^{j,k}=w_{x_k}^j-w_{x_j}^k$, similarly to the proof of
\cite[Lemma 2.3]{HalfSpace}, we have
$$\|\nabla\tilde{\tilde F}\|_p+\|\nabla\tilde{\tilde \omega}\|_p\leq C(\|\rho\dot w\|_p+\|\nabla w\|_p+\|w\|_p+\|\rho f\|_p),$$ i.e.
$\|\nabla\tilde{\tilde{F}}\|_p = \|\nabla((\lambda+\mu)\operatorname*{div} w-(P-\tilde P))\|_p
 \leq  C(\|\rho\dot w\|_p+\|\rho f\|_p+\|\nabla w\|_p+\|w\|_p)$.
Thence, following \cite[Lemma 3.3]{Xiangdi}, we obtain
$$
\begin{array}{rl}
&I_2  =  -\int_{\mathbb{R}_+^3} P_tw_j^jdx =  -\int_{\mathbb{R}_+^3} P'(\rho)\rho_t w_j^jdx
 =  \int_{\mathbb{R}_+^3} P'(\rho)\operatorname*{div}(\rho u)w_j^jdx\\
=&  \int_{\mathbb{R}_+^3} P'(\rho)(u.\nabla\rho) w_j^jdx+\int_{\mathbb{R}_+^3} P'(\rho)\rho \operatorname*{div} u \operatorname*{div} w dx\\
 \leq &  \int_{\mathbb{R}_+^3}\nabla(P-\tilde{P})u \operatorname*{div} wdx+C\int_{\mathbb{R}_+^3}|\nabla u\|\nabla w|dx\\
 = &  \int_{\mathbb{R}_+^3} \operatorname*{div}((P-\tilde P)u)\operatorname*{div} w dx-\int_{\mathbb{R}_+^3}(P-\tilde P)(\operatorname*{div} u)(\operatorname*{div} w)dx\\
 & \ \ \ \ \ \ \ \ \ \ \ + \ C\int_{\mathbb{R}_+^3}|\nabla u\|\nabla w|dx\\
  \leq &  -\int_{\mathbb{R}_+^3}(P-\tilde P)u.\nabla(\operatorname*{div} w)dx+C\int_{\mathbb{R}_+^3}|\nabla u\|\nabla w|dx\\
 = &  -\int_{\mathbb{R}_+^3}(P-\tilde P)u.\nabla (\operatorname*{div} w-\dfrac{(P-\tilde P)}{\lambda+\mu})dx-\int_{\mathbb{R}_+^3}(P-\tilde P)u.\nabla(\dfrac{P-\tilde P}{\lambda+\mu})dx\\
 & \ \ \ \ \ \ \ \ \ \ \ + \ C\int_{\mathbb{R}_+^3}|\nabla u\|\nabla w|dx\\
  \leq &  C\int_{\mathbb{R}_+^3}|u||\nabla(\operatorname*{div} w-\dfrac{(P-\tilde P)}{\lambda+\mu})|dx-\dfrac{1}{2(\lambda+\mu)}\int_{\mathbb{R}_+^3} u.\nabla\left((P-\tilde P)^2\right)dx\\
 & \ \ \ \ \ \ \ \ \ \ \ + \ C\int_{\mathbb{R}_+^3}|\nabla u\|\nabla w|dx\\
  = & C\int_{\mathbb{R}_+^3}|u||\nabla(\operatorname*{div} w-\dfrac{(P-\tilde P)}{\lambda+\mu})|dx + C\int_{\mathbb{R}_+^3} \operatorname*{div} u(P-\tilde P)^2dx\\
 & \ \ \ \ \ \ \ \ \ \ \ + \ C\int_{\mathbb{R}_+^3}|\nabla u\|\nabla w|dx\\
  \leq & C(C_0+C_f)^\theta\int_{\mathbb{R}_+^3}|\nabla(\operatorname*{div} w-\dfrac{(P-\tilde P)}{\lambda+\mu})|^2dx+C\int_{\mathbb{R}_+^3}|\nabla u|^2dx\\
 & \ \ \ \ \ \ \ \ \ \ \ + \ C\int_{\mathbb{R}_+^3} |\nabla w|^2dx\\
  \leq &  C(C_0+C_f)^\theta+C(C_0+C_f)^\theta\int_{\mathbb{R}_+^3}\rho|\dot w|^2dx+C\int_{\mathbb{R}_+^3} |\nabla u|^2dx\\
 & \ \ \ \ \ \ \ \ \ \ \ + \ C\int_{\mathbb{R}_+^3} |\nabla w|^2dx
\end{array}
$$

Regarding $I_3$, we have
$$
\begin{array}{ll}
&I_3  =  \int_{\mathbb{R}_+^3} \rho f^jw_t^jdx
 =  \int_{\mathbb{R}_+^3}\rho f^j(\dot w^j-u.\nabla w^j)dx\\
 \leq & C(C_0+C_f)^\theta\int_{\mathbb{R}_+^3}\rho|\dot w|^2dx\\
& \ \ \ + \  C(\int_{\mathbb{R}_+^3}\rho|f|^3dx)^{1/3}(\int_{\mathbb{R}_+^3}|u|^6dx)^{1/6}(\int_{\mathbb{R}_+^3}|\nabla w|^2dx)^{1/2}\\
 \leq & C(C_0+C_f)^\theta\int_{\mathbb{R}_+^3}\rho|\dot w|^2dx\\
& \ \ \ + \ C(C_0+C_f)^\theta (\int_{\mathbb{R}_+^3}|\nabla u|^2dx)^{1/2}(\int_{\mathbb{R}_+^3}|\nabla w|^2dx)^{1/2}
\end{array}
$$

Finally,
$$
\begin{array}{rl}
& I_4  =  \mu\int_{\partial\mathbb{R}_+^3} w_t^j(\nabla w^j).\nu dS(x)
  =  \mu\int_{\partial\mathbb{R}_+^3} w_t^jw_k^j\nu^kdS(x)\\
  = & -\mu\int_{\partial\mathbb{R}_+^3} w_t^jw_3^jdS(x)
 =  -\mu\int_{\partial\mathbb{R}_+^3} K^{-1} w_t^jw^jdS(x)\\
 = & -\dfrac{\mu}{2}\int_{\partial\mathbb{R}_+^3}(K^{-1}|w|^2)_t dS(x)
 =  -\dfrac{\mu}{2}\dfrac{d}{dt}\int_{\partial\mathbb{R}_+^3} K^{-1}|w|^2dS(x)
\end{array}
$$

Therefore
$$
\begin{array}{ll}
&\int_{\mathbb{R}_+^3} \rho|\dot w|^2dx  + (\dfrac{\mu}{2}\int_{\mathbb{R}_+^3}|\nabla w|^2dx+\dfrac{\lambda}{2}\int_{\mathbb{R}_+^3}|\operatorname*{div} w|^2dx\\
& \ \ \ \ \ \ \ \ \ \ \ \ \ \ \ \ \ \ \ \ \ \ \ \ - \ \int_{\mathbb{R}_+^3}(P-\tilde P)\operatorname*{div} w dx+\dfrac{\mu}{2}\int_{\partial\mathbb{R}_+^3} K|w|^2dS(x))_t\\
  \leq & C(C_0+C_f)^\theta + C(C_0+C_f)^\theta\int_{\mathbb{R}_+^3}|\nabla u|^2 dx+ C\int_{\mathbb{R}_+^3}|\nabla u\|\nabla w|dx\\
 & \ \ \ \ \ \  + \  C(C_0+C_f)^\theta\int_{\mathbb{R}_+^3}\rho|\dot w|^2dx+C\int_{\mathbb{R}_+^3}|\nabla w|^2dx\\
 & \ \ \ \ \ \ \ \ \ \ \ + \  C(C_0+C_f)^\theta(\int_{\mathbb{R}_+^3}|\nabla u|^2dx)^{1/2}(\int_{\mathbb{R}_+^3}|\nabla w|^2dx)^{1/2}.
\end{array}
$$
Integrating in $(0,t)$ and taking $C_0,C_f$ sufficiently small, we obtain
$$
\begin{array}{ll}
&\int_0^t\int_{\mathbb{R}_+^3}\rho|\dot w|^2dxds  +  \dfrac{\mu}{2}\int_{\mathbb{R}_+^3} |\nabla w|^2dx+\dfrac{\lambda}{2}\int_{\mathbb{R}_+^3}|\operatorname*{div} w|^2dx\\
& \ \ \ \ \ \ \ \ \ \ \ \ \ \ \ \ \ \ \ \ \ \ \ \ \ \ \ \ \ \ \ \ \ \ \ \ \ \ \ \ + \ \dfrac{\mu}{2}\int_{\partial\mathbb{R}_+^3} K|w|^2dS(x)\\
 \leq & C(C_0+C_f)^\theta + \int_{\mathbb{R}_+^3}(P-\tilde P)(\operatorname*{div} w)dx\\
 & \ \ \ \ \ \ + CM_q(\int_0^t\int_{\mathbb{R}_+^3}|\nabla u|^2dxds)^{1/2}(\int_0^t\int_{\mathbb{R}_+^3} |\nabla w|^2dxds)^{1/2}\\
 \leq & C(C_0+C_f)^\theta +C(C_0+C_f)^\theta \int_{\mathbb{R}_+^3} |\nabla w|^2dx,
\end{array}
$$
hence obtaining the result, assuming again $C_0,C_f$ sufficiently small.
\end{proof}

Now we are ready to show \eqref{expoente1}, i.e. we have

\begin{theorem} The estimate \eqref{expoente1} holds for Hoff \cite{HalfSpace} solutions $(\rho,u)$ of \eqref{equations int}-\eqref{initial},
if $u_0\in H^s(\mathbb{R}_+^3)$, $0\le s\le 1$.
\label{theorem 1-s}
\end{theorem}
\begin{proof}
Let $v$ and $w$ be the solutions of
\eqref{homogeneo} and \eqref{nao-homogeneo}), respectively. Since $v|_{t=0}=u_0$, by the unicity of solution of the linear system $\mathcal{L}(z)=\nabla(P-\tilde P) + \rho f$, joint with the initial condition $z_{t=0}=u_0$, we have that $u=v+w$. (Notice that $z=v+w$ and $z=u$
are both solutions of this problem.) Thus, by \ref{interp_para_v} and \ref{interp_para_w}, we obtain (\ref{expoente1}). \end{proof}

Next we shall show the estimate \eqref{expoente2}. Now we do not need to split the solution $u$ as the sum of two other functions 
 as we did above to obtain \eqref{expoente1}, but to show \eqref{expoente2} we shall use \eqref{expoente1}.

\begin{theorem} The estimate \eqref{expoente2} holds for Hoff \cite{HalfSpace} solutions $(\rho,u)$ of \eqref{equations int}-\eqref{initial},
if $u_0\in H^s(\mathbb{R}_+^3)$, $s>\frac{1}{2}$. 
\label{theorem 2-s}
\end{theorem}
\begin{proof}
Writing the momentum equation as $\rho\dot u^j+P_j=\mu\Delta u^j+\lambda \operatorname*{div} u_j+\rho f^j$,
and applying the operator $\sigma^m\dot u^j\left(\partial_t(.)+\operatorname*{div}(.u)\right)$, $m\geq 1$,
as in \cite{Hoff2002} and \cite{Xiangdi}, we have
$$
\begin{array}{ll}
&\sigma^m\rho \dot u_t^j\dot u^j+\sigma^m\rho u.\nabla\dot u^j\dot u^j + \sigma^m\dot u^jP_{jt}+\sigma^m\dot u^j \operatorname*{div}(P_j u)\\
 = & \mu\sigma^ m\dot u^j(\Delta u_t^j+\operatorname*{div}(u\Delta u^j))+\lambda\sigma^m\dot u^j(\partial_t\partial_j \operatorname*{div} u+ \operatorname*{div}(u\partial_j \operatorname*{div} u))\\
& \ \ \ \ \ \ \ \ \ \ \ \ \ \ \ \ \ \ \ \ \ \ \ \ \ \ \ \ \ \ \
  +  \sigma^m\rho\dot u^jf_t^j+\sigma^m\rho u^kf_k^j.
\end{array}
$$
Notice that
$$
\begin{array}{ll}
  &\sigma^m\rho\dot u_t^j\dot u^j+\sigma^m\rho u.\nabla\dot u^j\dot u^j  =  \dfrac{\sigma^m}{2}\left(\rho\partial_t(|\dot u|^2)+\rho u.\nabla(|\dot u|^2)\right)\\
= & \partial_t\left(\dfrac{\sigma^m}{2}\rho|\dot u|^2\right)-\dfrac{m}{2}\sigma^{m-1}\sigma'\rho|\dot u|^2-\dfrac{\sigma^m}{2}\rho_t|\dot u|^2+\dfrac{\sigma^m}{2}\rho u.\nabla(|\dot u|^2).
\end{array}
$$
Integrating in ${\mathbb{R}_+^3}$, it follows that
$$
\begin{array}{ll}
&(\dfrac{\sigma^m}{2}\int_{\mathbb{R}_+^3}\rho|\dot u|^2dx)_t-\dfrac{m}{2}\sigma'\sigma^{m-1}\int_{\mathbb{R}_+^3}\rho|\dot u|^2dx\\
= & -\sigma^m\int_{\mathbb{R}_+^3}\dot u^j\left(P_{jt}+\operatorname*{div}(P_j u)\right)dx
 +  \mu\sigma^m\int_{\mathbb{R}_+^3}\dot u^j(\Delta u_t^j+\operatorname*{div}(u\Delta u^j))dx\\
& +  \lambda\sigma^m\int_{\mathbb{R}_+^3}\dot u^j(\partial_t\partial_j \operatorname*{div} u+ \operatorname*{div}(u\partial_j \operatorname*{div} u))dx +  \sigma^m\int_{\mathbb{R}_+^3}(\rho\dot u^jf_t^j+\rho u^kf_k^j)dx\\
\equiv & \sum_{i=1}^4N_i.
\end{array}
$$
Let us estimate each of this terms $N_1,N_2,N_3,N_4$, separately.
 \ Integrating by parts,
$$
\begin{array}{ll}
&N_1=-\int_{\mathbb{R}_+^3}\sigma^m\dot u^j\left(\partial_tP_j+\operatorname*{div}(P_j u)\right)dx\\
 = & \sigma^ m\int_{\mathbb{R}_+^3}\dot u_j^jP'\rho_tdx-\int_{\partial{\mathbb{R}_+^3}}\sigma^ m\dot u^j\nu^jP_tdS_x+\sigma^ m\int_{\mathbb{R}_+^3} \dot u_k^jP_ju^k\\
& \ \ \ \ \ \ \ \ \ \ \ \ \ \ \ \ \ \ \ \ \ \ \ \ \ \ \ \ \ \ \ \ \ \ \ \ \ \ \ - \ \int_{\partial{\mathbb{R}_+^3}}\sigma^ m\dot u^jP_ju.\nu dS_x\\
= & \sigma^m\int_{\mathbb{R}_+^3}\dot u_j^jP'(-\rho \operatorname*{div} u-u.\nabla\rho)dx+\sigma^ m\int_{\mathbb{R}_+^3}\dot u_k^jP_ju^kdx\\
= & -\sigma^m\int_{\mathbb{R}_+^3} P'\rho\dot u_j^j\operatorname*{div} udx-\sigma^m\int_{\mathbb{R}_+^3}\dot u_j^ju.\nabla Pdx\\
 & - \ \sigma^m\int_{\mathbb{R}_+^3} (P-\tilde P)(\dot u_{jk}^ju^k+\dot u_k^ju_j^k)dx+\sigma^m\int_{\partial{\mathbb{R}_+^3}}(P-\tilde P)\dot u_k^ju^k\nu^jdS_x\\
 = & -\sigma^m\int_{\mathbb{R}_+^3} P'\rho\dot u_j^j\operatorname*{div} udx+\sigma^m\int_{\mathbb{R}_+^3} (P-\tilde P)(\dot u_{jk}^ju^k+\dot u_j^ju_k^k)dx\\
 & -  \sigma^m\int_{\partial{\mathbb{R}_+^3}}(P-\tilde P)(\dot u_j^ju.\nu)dS_x-\sigma^m\int_{\mathbb{R}_+^3} (P-\tilde P)(\dot u_{jk}^ju^k+\dot u_k^ju_j^k)dx\\
  = & -\sigma^ m\int_{\mathbb{R}_+^3}\left(P'\rho\dot u_j^j\operatorname*{div} u-P\dot u_j^ju_k^k+P\dot u_k^ju_j^k\right)dx\\
  \leq & C(\bar\rho)\sigma^m\|\nabla u\|_2\|\nabla\dot u\|_2\\
  \leq & C(\bar\rho)C(\varepsilon)\sigma^m\|\nabla u\|_2^2+C(\bar\rho)\varepsilon\sigma^m\|\nabla\dot u\|_2^2.
\end{array}
$$

$$
\begin{array}{ll}
& N_2  =  \mu\sigma^m\int_{\mathbb{R}_+^3}\dot u^j(\Delta u_t^j+\operatorname*{div}(u\Delta u^j))dx\\
= & \mu\sigma^m\int_{\mathbb{R}_+^3}(\dot u^ju_{kkt}^j+\dot u^j(u^ku_{ll}^j)_k)dx\\
= & -\mu\sigma^m\{\int_{\mathbb{R}_+^3}(\dot u_k^ju_{kt}^j)dx-\int_{\partial{\mathbb{R}_+^3}}\dot u^ju_{kt}^j\nu^k dS_x\}\\
  - & \mu\sigma^m\{\int_{\mathbb{R}_+^3}\dot u_k^ju^ku_{ll}^jdx-\int_{\partial{\mathbb{R}_+^3}}\dot u^ju^ku_{ll}^j\nu^k dS_x\}\\
= & -\mu\sigma^m\{\int_{\mathbb{R}_+^3} \dot u_k^j(\dot u_k^j-(u.\nabla u^j)_k)dx-\int_{\partial{\mathbb{R}_+^3}}\dot u^ju_{kt}^j\nu^k dS_x\}\\
  & + \mu\sigma^m\{\int_{\mathbb{R}_+^3}(\dot u_{kl}^ju^ku_l^j+\dot u_k^ju_l^ku_l^j)dx-\int_{\partial{\mathbb{R}_+^3}}\dot u_k^ju^ku_l^j\nu^l dS_x\}\\
= & -\mu\sigma^m\{\int_{\mathbb{R}_+^3}|\nabla\dot u|^2dx-\int_{\mathbb{R}_+^3}\dot u_k^ju_k^lu_l^jdx-\int_{\mathbb{R}_+^3} \dot u_k^ju^lu_{kl}^jdx\\
& \ \ \ \ \ \ \ \ \ \ \ \ \ \ \ \ \ \ \ \ \ \ \ \ \ \ \ - \ \int_{\partial{\mathbb{R}_+^3}} \dot u^ju_{kt}^j\nu^k dS_x\}\\
  & + \ \mu\sigma^m\{\int_{\mathbb{R}_+^3}(\dot u_{kl}^ju^ku_l^j+\dot u_k^ju_l^ku_l^j)dx-\int_{\partial{\mathbb{R}_+^3}}\dot u_k^ju^ku_l^j\nu^l dS_x\}\\
 = &  -\mu\sigma^m\{\int_{\mathbb{R}_+^3}|\nabla\dot u|^2dx-\int_{\mathbb{R}_+^3}\dot u_k^ju_k^lu_l^jdx+\int_{\mathbb{R}_+^3}( \dot u_{kl}^ju^lu_{k}^j+\dot u_k^ju_l^lu_k^j)dx\\
 & \ \ \ \ \ \ \ \ \ \ \ \ \ \ \ \ \ \ \ \ \ \ \ \ \ \ \ - \  \int_{\partial{\mathbb{R}_+^3}} \dot u^ju_{kt}^j\nu^k dS_x\}\\
  & +  \mu\sigma^m\{\int_{\mathbb{R}_+^3}(\dot u_{kl}^ju^ku_l^j+\dot u_k^ju_l^ku_l^j)dx-\int_{\partial{\mathbb{R}_+^3}}\dot u_k^ju^ku_l^j\nu^l dS_x\}\\
  = & -\mu\sigma^m\left\{\int_{\mathbb{R}_+^3}|\nabla\dot u|^2dx-\int_{\mathbb{R}_+^3}(\dot u_k^ju_k^lu_l^j-\dot u_k^ju_l^lu_k^j+\dot u_k^ju_l^ku_l^j)dx\right\}\\
  & +  \mu\sigma^m\int_{\partial{\mathbb{R}_+^3}}(\dot u^ju_{kt}^j\nu^k-\dot u_k^ju^ku_l^j\nu^l)dS_x;
\end{array}
$$
to estimate the boundary term above, we write
$$
\mu\sigma^m\int_{\partial{\mathbb{R}_+^3}}(\dot u^ju_{kt}^j\nu^k-\dot u_k^ju^ku_l^j\nu^l)dS_x\equiv N_{21}+N_{22},
$$
using that if $h\in (C^1\cap W^{1,1})(\overline{\mathbb{R}_+^3})$ then
$\int_{\partial{\mathbb{R}_+^3}}h(x)dS_x=\int_{\{0\leq x_3\leq 1\}}[h(x)+(x_3-1)h_{x_3}(x)]dx$. Observe
that we can assume $j\neq 3$ in $N_{21}$ and $k\neq 3$ in $N_{22}$ without loss of generality, since $u^3=0$ on $\partial{\mathbb{R}_+^3}$. Let us show how to estimate the term $N_{21}$ above. The term $N_{22}$ can be estimate similarly.
$$
\begin{array}{ll}
& N_{21} =  -\mu\sigma^m\int_{\partial{\mathbb{R}_+^3}}\dot u^ju_{3t}^jdS_x
=  -\mu\sigma^m\int_{\partial{\mathbb{R}_+^3}}K^{-1}\dot u^ju_t^j dS_x\\
  = & -\mu\sigma^m\int_{\partial{\mathbb{R}_+^3}}K^{-1}\dot u^j(\dot u^j-u^ku_k^j)dS_x\\
  = & -\mu\sigma^m\int_{\partial{\mathbb{R}_+^3}}K^{-1}|\dot u|^2dS_x + \mu\sigma^m\int_{\partial{\mathbb{R}_+^3}}K^{-1}\dot u^ju^ku_k^jdS_x\\
 \leq & \mu\sigma^m\int_{\partial{\mathbb{R}_+^3}}K^{-1}\dot u^ju^ku_k^jdS_x\\
  = & \mu\sigma^m\int_{\{0\leq x_3\leq 1\}}K^{-1}(\dot u^ju^ku_k^j+(x_3-1)[\dot u_3^ju^ku_k^j+\dot u^ju_3^ku_k^j+\dot u^ju^ku_{k3}^j])dx\\
  \leq & C \mu\sigma^m\int_{\mathbb{R}_+^3}(|\nabla\dot u\|u\|\nabla u|+|\dot u\|\nabla u\|u|+|\dot u\|\nabla u|^2) dx\\
 & -  \mu\sigma^m\int_{\{0\leq x_3\leq 1\}}(x_3-1)(K^{-1}\dot u_k^ju^ku_3^j+K^{-1}\dot u^ju_k^ku_3^j+(K^{-1})_k \dot u^ju^ku_3^j)dx\\
 & +  \mu\sigma^m\int_{\{x_3=0\}\cup\{x_3=1\}}K^{-1}(x_3-1)\dot u^ ju^ ku_3^j\nu^ kdS_x.
\end{array}
$$
Notice the above boundary term is null, since for $x_3=0$ we have $u^k\nu^k=0$ and for $x_3=1$ the
term $(x_3-1)$ vanishes the integrand. Thus,
$$
N_{21}\leq C \sigma^m\int_{\mathbb{R}_+^3} (|\nabla\dot u\|u\|\nabla u|+|\dot u\|\nabla u\|u|+|\dot u\|\nabla u|^2)dx.
$$

Regarding $N_3$, setting $D=\operatorname*{div} u$, we have
$$
\begin{array}{ll}
& N_3  =  \lambda\sigma^m\int_{\mathbb{R}_+^3} \dot u^j\left(\partial_t\partial_j\operatorname*{div} u+\operatorname*{div}(u\partial_j\operatorname*{div} u)\right)dx\\
= & -\lambda\sigma^m\int_{\mathbb{R}_+^3}(\dot u_j^jD_t)dx+\lambda\sigma^m\int_{\partial{\mathbb{R}_+^3}}\dot u^jD_t\nu^jdS_x\\
& \ \ \ \ \ \ \ \ \ \ \ \ \ \ \ \ \ \ \ + \ \lambda\sigma^m\int_{\mathbb{R}_+^3}\dot u^j(DD_j+u^kD_{jk})dx\\
 = &  -\lambda\sigma^m\int_{\mathbb{R}_+^3}(\dot u_j^jD_t)dx+\lambda\sigma^m\int_{\mathbb{R}_+^3}\dot u^jDD_jdx+\int_{\mathbb{R}_+^3}\dot u^ju^kD_{jk}dx\\
 \equiv & N_{31}+N_{32}+N_{33}.
\end{array}
$$
Notice that $N_{31}=-\lambda\sigma^m\int_{\mathbb{R}_+^3} \dot u_j^jD_tdx=-\lambda\sigma^m\int_{\mathbb{R}_+^3} \dot u_j^j\dot Ddx+\lambda\sigma^m\int_{\mathbb{R}_+^3}\dot u_j^ju\cdot\nabla Ddx$. For $N_{32}$ we have
$$\begin{array}{ll}
& N_{32} =  \dfrac{\lambda}{2}\sigma^m\int_{\mathbb{R}_+^3} \dot u^j(|D|^2)_jdx\\
 = & -\dfrac{\lambda}{2}\sigma^m\int_{\mathbb{R}_+^3}\dot u_j^j|D|^2dx+\dfrac{\lambda}{2}\sigma^ m\int_{\partial{\mathbb{R}_+^3}}\dot u^j\nu^j|D|^2dS_x\\
 \leq & C\sigma^m\int_{\mathbb{R}_+^3}|\nabla\dot u\|\nabla u|^2dx\\
 \leq & C\varepsilon\sigma^m\int_{\mathbb{R}_+^3}|\nabla\dot u|^2dx+C\sigma^m\int_{\mathbb{R}_+^3} |\nabla u|^4dx.
\end{array}$$

$$\begin{array}{ll}
&N_{33}=\lambda\sigma^m\int_{\mathbb{R}_+^3} u^k\dot u^jD_{jk}dx\\
 = & -\lambda\sigma^m\int_{\mathbb{R}_+^3}(\dot u_j^ju^kD_k+\dot u^ju_j^kD_k)dx+\lambda\sigma^m\int_{\partial{\mathbb{R}_+^3}}D_ku^k\dot u^j\nu^jdS_x\\
= & -\lambda\sigma^m\int_{\mathbb{R}_+^3}\dot u_j^ju.\nabla Ddx+\lambda\sigma^m\int_{\mathbb{R}_+^3}(\dot u_k^ju_j^kD+\dot u^ju_{kj}^kD)dx\\
& \ \ \ \ \ \ \ \ \ \ \ \ \ \ \ \ \ \ \ \ \ - \ \lambda\sigma^m\int_{\partial{\mathbb{R}_+^3}}\dot u^ju_j^ku_l^l\nu^kdS_x\\
 = &-\lambda\sigma^m\int_{\mathbb{R}_+^3}\dot u_j^ju.\nabla Ddx+\lambda\sigma^m\int_{\mathbb{R}_+^3} \dot u_k^ju_j^kDdx-\dfrac{\lambda}{2}\sigma^m\int_{\mathbb{R}_+^3}\dot u_j^j|D|^2dx\\
 & \ \ \ \ \ \ \ \ \ \ \ \ \ \ \ \ \ \ \ \ \ + \ \dfrac{\lambda}{2}\sigma^m\int_{\partial{\mathbb{R}_+^3}}\dot u^j\nu^j|D|^2dS_x.
\end{array}
$$
Thus,
$$
\begin{array}{ll}
& N_{31}+N_{33}\\
=& -\lambda\sigma^m\int_{\mathbb{R}_+^3}\dot u_j^j\dot Ddx+\lambda\sigma^m\int_{\mathbb{R}_+^3}\dot u_k^ju_j^kDdx-\dfrac{\lambda}{2}\sigma^m\int_{\mathbb{R}_+^3}\dot u_j^j|D|^2dx\\
\leq & -\lambda\sigma^m\int_{\mathbb{R}_+^3}\dot u_j^j\dot Ddx+\varepsilon\sigma^m\int_{\mathbb{R}_+^3}|\nabla\dot u|^2dx+C\sigma^m\int_{\mathbb{R}_+^3}|\nabla u|^4dx\\
= & -\lambda\sigma^m\int_{\mathbb{R}_+^3}\dot D(u_t^j+u.\nabla u^j)_jdx+\varepsilon\sigma^m\int_{\mathbb{R}_+^3}|\nabla\dot u|^2dx\\
& \ \ \ \ \ \ \ \ \ \ \ \ \ \ \ \ \ \ \ \ \ + \ C\sigma^m\int_{\mathbb{R}_+^3}|\nabla u|^4dx\\
= & -\lambda\sigma^m\int_{\mathbb{R}_+^3}\dot D(D_t+u.\nabla D+u_j^ku_k^j)dx+\varepsilon\sigma^m\int_{\mathbb{R}_+^3}|\nabla\dot u|^2dx\\
& \ \ \ \ \ \ \ \ \ \ \ \ \ \ \ \ \ \ \ \ \ + \ C\sigma^m\int_{\mathbb{R}_+^3}|\nabla u|^4dx\\
= & -\lambda\sigma^m\int_{\mathbb{R}_+^3}|\dot D|^2dx-\lambda\sigma^m\int_{\mathbb{R}_+^3}\dot D u_j^ku_k^jdx
 +\varepsilon\sigma^m\int_{\mathbb{R}_+^3}|\nabla\dot u|^2dx\\
 & \ \ \ \ \ \ \ \ \ \ \ \ + \ C\sigma^m\int_{\mathbb{R}_+^3}|\nabla u|^4dx\\
 \leq & -\lambda\sigma^m\int_{\mathbb{R}_+^3}|\dot D|^2dx+\varepsilon\sigma^m\int_{\mathbb{R}_+^3}|\nabla\dot u|^2dx+C\sigma^m\int_{\mathbb{R}_+^3}|\nabla u|^4dx.
\end{array}
$$

Finally,
$$
\begin{array}{ll}
& N_4  =  \sigma^m\int_{\mathbb{R}_+^3}(\rho\dot u^jf_t^j+\rho\dot u^ju^kf_k^j)dx\\
 \leq & \varepsilon\sigma^{m-1}\int_{\mathbb{R}_+^3}\rho |\dot u|^2dx+C \sigma^{m+1}\int_{\mathbb{R}_+^3}|f_t|^2dx\\
 & +  \varepsilon\sigma^{m}\int_{\mathbb{R}_+^3}\rho|\dot u|^2dx+ C \sigma^{m+1}\int_{\mathbb{R}_+^3}|\nabla f|^2|u|^2dx\\
  \leq & \varepsilon\sigma^{m-1}\int_{\mathbb{R}_+^3}\rho|\dot u|^2dx+C\sigma^{m+1}\int_{\mathbb{R}_+^3}|f_t|^2dx\\
 & +  C (\sigma^{(3-3s)/2}\int_{\mathbb{R}_+^3}|u|^4dx)^{1/2}(\sigma^{(4m+1+3s)/2}\int_{\mathbb{R}_+^3}|\nabla f|^4dx)^{1/2}\\
  \leq & \varepsilon\sigma^{m-1}\int_{\mathbb{R}_+^3}\rho|\dot u|^2dx+C\sigma^{m+1}\int_{\mathbb{R}_+^3}|f_t|^2dx\\
 & +  C(C_0+C_f)^\theta(\sigma^{(4m+1+3s)/2}\int_{\mathbb{R}_+^3}|\nabla f|^4dx)^{1/2},
\end{array}
$$
since
$
\sigma^{(3-3s)/2}\|u\|_4^4 \leq  C\sigma^{(3-3s)/2}\|u\|_2\|\nabla u\|_2^3
 \leq  C(C_0+C_f)^\theta\sigma^{(3-3s)/2}\|\nabla u\|_2^3
  =  C(C_0+C_f)^\theta(\int_{\mathbb{R}_+^3} \sigma^{1-s} |\nabla u|^2dx)^{3/2}
  \leq  C(C_0+C_f)^\theta.
$

With these estimates, we arrive at
$$
\begin{array}{ll}
&(\dfrac{\sigma^m}{2}\int_{\mathbb{R}_+^3}\rho|\dot u|^2dx)_t  -  \dfrac{m}{2}\sigma'\sigma^{m-1}\int_{\mathbb{R}_+^3}\rho|\dot u|^2dx\\
 \leq & C(\bar\rho)C(\varepsilon)\sigma^ m\int_{\mathbb{R}_+^3} |\nabla u|^2dx+C(\bar\rho)\varepsilon\sigma^m\int_{\mathbb{R}_+^3}|\nabla \dot u|^2dx\\
& -  \mu\sigma^m\int_{\mathbb{R}_+^3}|\nabla\dot u|^2dx+C\sigma^m\int_{\mathbb{R}_+^3}|\nabla u|^4dx\\
& +  C\varepsilon\sigma^m\int_{\mathbb{R}_+^3}|\nabla\dot u|^2dx+C\sigma^m\int_{\mathbb{R}_+^3}|\nabla u|^4dx-\lambda\sigma^m\int_{\mathbb{R}_+^3} |\dot D|^2dx\\
& +  C\sigma^m\int_{\mathbb{R}_+^3}|\dot u|^2dx+ C\sigma^m\int_{\mathbb{R}_+^3}|u|^4dx\\
& +  \varepsilon\sigma^{m-1}\int_{\mathbb{R}_+^3}\rho|\dot u|^2dx+C\sigma^{m+1}\int_{\mathbb{R}_+^3}|f_t|^2dx\\
 & +  C(C_0+C_f)^\theta\left(\sigma^{(4m+1+3s)/2}\int_{\mathbb{R}_+^3}|\nabla f|^4dx\right)^{1/2}.
\end{array}
$$
Integrating in $(0,T)$, taking $m=2-s$ and using \eqref{expoente1}, we obtain
$$
\begin{array}{ll}
&\sigma^m\int_{\mathbb{R}_+^3} \rho|\dot u|^2dx  +  \int_0^T\int_{\mathbb{R}_+^3}\sigma^m|\nabla\dot u|^2dxds\\
  \leq & C(C_0+C_f)^\theta+C\int_0^T\int_{\mathbb{R}_+^3}\sigma^m(|\nabla u|^4+|u|^4)dxds\\
 & +  \int_0^T\sigma^{3-s}\int_{\mathbb{R}_+^3}|f_t|^2dxdt+\int_0^T\sigma^{(9-s)/4}\left(\int_{\mathbb{R}_+^3}|\nabla f|^4dx\right)^{1/2}dt\\
  \leq & C(C_0+C_f)^\theta+C\int_0^T\int_{\mathbb{R}_+^3}\sigma^m(|\nabla u|^4+|u|^4)dxdt.
\end{array}
$$

To conclude the result we must estimate the term $\int_0^T\sigma^{2-s}\int_{\mathbb{R}_+^3}(|\nabla u|^4+|u|^4)dxd\tau$. Using \eqref{des_interpolacao}, we have
$$
\begin{array}{ll}
&\int_0^T\sigma^{2-s}\|u\|_4^4d\tau  \leq  \int_0^T\sigma^{2-s}\|u\|_2\|\nabla u\|_2^3d\tau\\
& =  \int_0^T\sigma^{2-s}(\int_{\mathbb{R}_+^3}|u|^2dx)^{1/2}(\int_{\mathbb{R}_+^3}|\nabla u|^2dx)^{3/2}d\tau\\
&\leq  C(C_0+C_f)^\theta\int_0^T\sigma^{2-s}(\int_{\mathbb{R}_+^3}|\nabla u|^2dx)^{3/2}d\tau\\
&\leq C(C_0+C_f)^\theta\int_0^T\sigma^{\frac{1+s}{2}}(\sigma^{1-s}\int_{\mathbb{R}_+^3}|\nabla u|^2dx)^{3/2}d\tau\\
&\leq C(C_0+C_f)^\theta.
\end{array}
$$
On the other hand, following \cite[Lemma 3.3]{Xiangdi}, using energy estimates and \eqref{expoente1}, we can estimate
$$
\begin{array}{ll}
     & \int_0^T\sigma^{2-s}\int_{\mathbb{R}_+^3}|\nabla u|^4dxd\tau  =  \int_0^T\sigma^{2-s}\|\nabla u\|_4^4d\tau\\
\leq & C\int_0^T\sigma^{2-s}\|\nabla u\|_2\left(\|\rho\dot u\|_2+\|\nabla u\|_2+\|u\|_2+\|f\|_2+\|P-\tilde P\|_6\right)^3d\tau\\
\leq & C\int_0^T\sigma^{2-s}\|\nabla u\|_2\left(\|\rho\dot u\|_2^3+\|\nabla u\|_2^3+\|u\|_2^3+\|f\|_2^3+\|P-\tilde P\|_6^3\right)d\tau\\
\leq & C(C_0+C_f)^\theta+C\int_0^T\sigma^{2-s}\|\nabla u\|_2\|\rho\dot u\|_2^3d\tau\\
 \leq & C(C_0+C_f)^\theta\\
 & +C\int_0^T\sigma^{2-s}(\int_{\mathbb{R}_+^3}|\nabla u|^2dx)^{1/2}(\int_{\mathbb{R}_+^3}\rho|\dot u|^2dx)^{1/2}(\int_{\mathbb{R}_+^3}\rho|\dot u|^2dx)d\tau\\
 \leq & C(C_0+C_f)^\theta +
   C\int_0^T\sigma^{\frac{2s-1}{2}}\\
   & \left(\sigma^{1-s}\int_{\mathbb{R}_+^3}|\nabla u|^2dx)^{1/2}(\sigma^{2-s}\int_{\mathbb{R}_+^3}\rho|\dot u|^2dx)^{1/2}(\sigma^{1-s}\int_{\mathbb{R}_+^3}\rho|\dot u|^2dx\right)d\tau\\
   \leq & C(C_0+C_f)^\theta +C(C_0+C_f)^\theta\sup\limits_{0\leq t\leq T}(\sigma^{2-s}\int_{\mathbb{R}_+^3}\rho|\dot u|^2dx)^{1/2}.
\end{array}
$$
Therefore, using once more that $C_0,C_f$ are sufficiently small, we obtain \eqref{expoente2}.
\end{proof}

\subsection{Proof of Theorem \ref{main}}
\label{proof of main}

The proof of Theorem \ref{main} using theorems \ref{decomposition}, \ref{main 1}
is similar to the proof of Theorem 2.5 in \cite{Santos}. Thus, here we
just give an overview of it and show some steps which may be peculiar to our case.

The proof of existence of the particle paths $X(t,x)$ satisfying \eqref{lagrangean_structure} is obtained estimating
$|X(t_1,x)-X(t_2,x)|  \leq \int_{t_1}^{t_2}||u(t,\cdot)||_\infty dt
 \leq  C\int_{t_1}^{t_2}(||u(t,\cdot)||_2+||\nabla u(t,\cdot)||_p)dt$, $p>3$,
uniformly with respect to smooth solutions, where for the last inequality we
used \eqref{limitacao}. Then proceeding with several estimates, it is possible
to show that $X(t,x)$ is H\"older continuous in $t$, uniformly with respect
to smooth solutions. Its uniqueness follows from the estimate
$\int_0^T \langle u(.,t)\rangle_{LL}dt\leq CT^\gamma$, for any $T>0$, where $C$ and $\gamma$, are positive constants, possibly depending on $T$, but not
depending on $u$, cf. \cite[lemmas 3.1 and 3.2]{Santos}
(or see \cite[Lema 3.3]{Teixeira}).

For the proof of the item (b) of Theorem \ref{main}, first we observe that
the injectivity and openness of the map $x\mapsto X(t,x)$ can be shown exactly
as in \cite{Santos}. To show the surjectivity, we use the particles paths starting at $t_0>0$, i.e. the map $X(.,x_0;t_0)\in C([0,+\infty),\overline{\mathbb{R}_+^3})\cap C^ 1((0,+\infty);\overline{\mathbb{R}_+^3})$ such that
$X(t,x_0;t_0)=x_0+\int_{t_0}^t u(X(\tau,x_0),\tau)d\tau$
(see \cite[Corollary 2.3]{Santos} or \cite[Teorema 3.4]{Teixeira}):
given $y\in\overline{\mathbb{R}_+^3}$, let $Y(s)=X(s;y,t)$, $s\in[0,t]$.
Since the curves $Y(s)$ e $X(s,Y(0))$ satisfy
$Z'(s)=u(Z(s),s), \, Z(0)=Y(0)$, we have $Y(s)=X(s;Y(0)), s\in[0,t]$, so $y=Y(t)=X(t;Y(0))$, which shows the surjectivity of the map $X(t,.):\overline{\mathbb{R}_+^3}\to\overline{\mathbb{R}_+^3}$. The continuity is a direct consequence of item (c).
To show the invariance of the boundary $\partial\mathbb{R}_+^3$ by the flux, let
$x=(x_1,x_2,0)\in\partial\mathbb{R}_+^3$. Defining $X^i(\cdot, x)$, for $i=1,2$, by
$X^i(t,x)=x_i+\int_0^t u^i(X^i(\tau,x),\tau)d\tau$, we have that
$Y(t,x):=(X^1(t,x),X^2(t,x),0)$ is a path which lies in
$\partial\mathbb{R}_+^3$ and verifies $dY(t,x)/dt=u(Y(t,x),t), \, t>0,  Y(0)=x$, since $u^3=0$ on $\partial\mathbb{R}_+^3$, so, by uniqueness (item (a)) we have $Y(t,x)=X(t,x)$, for all $t\ge0$, and thus
we conclude the invariance of the boundary by the flux $X(t,\cdot)$.

The proofs of the items (c) and (d) can be done exactly as the proofs of
\cite[Theorem 2.5 (c),(d)]{Santos} (or \cite[Teorema 3.2 (d),(f)]{Teixeira}).
Since the proof of item (d), using item (c), is very simple, for the convenience
of the reader we present it here:  Let $\mathcal{M}$ be as in the statement of item (d).
Then, by definition, $\mathcal{M}$ is the image of a map $\psi:U\to\mathbb{R}_+^3$ of
class $C^\alpha$, where $U$ is an open set of $\mathbb{R}^k$, $k=1$ or $2$.
Then $\varphi_t(x)=X(t,\psi(x))$ is a parametrization of $X(t,\mathcal{M})$ and, by item (c), we have
$|\varphi_t(x)-\varphi_t(y)|=|X(t,\psi(x))-X(t,\psi(y))|
\leq  C|\psi(x)-\psi(y)|^{e^{-LT^\gamma}}
\leq  C|x-y|^{\alpha e^{-LT^\gamma}}$,
which ends the proof.  \ \fbox

\bibliographystyle{amsplain}

\end{document}